\documentclass{amsproc}
\usepackage{euscript}
\usepackage{cases}
\usepackage{mathrsfs}
\usepackage{bbm}
\usepackage{amssymb}
\usepackage{amsfonts,amsmath,amsxtra,mathdots,mathabx}
\usepackage{color}
\usepackage{hyperref}
\usepackage{tikz}
\usepackage{appendix,upgreek}

\allowdisplaybreaks

\DeclareFontFamily{U}{matha}{\hyphenchar\font45}
\DeclareFontShape{U}{matha}{m}{n}{
	<5> <6> <7> <8> <9> <10> gen * matha
	<10.95> matha10 <12> <14.4> <17.28> <20.74> <24.88> matha12
}{}
\DeclareSymbolFont{matha}{U}{matha}{m}{n}

\DeclareMathSymbol{\Lt}{3}{matha}{"CE}
\DeclareMathSymbol{\Gt}{3}{matha}{"CF}

\DeclareSymbolFont{mathc}{OML}{txmi}{m}{it}
\DeclareMathSymbol{\varuu}{\mathord}{mathc}{117}
\DeclareMathSymbol{\varvv}{\mathord}{mathc}{118}
\DeclareMathSymbol{\varww}{\mathord}{mathc}{119}

\def\valpha{\text{\scalebox{0.88}[1.02]{$\alpha$}}}   
\def\vepsilon{\upvarepsilon}

\def\vchi{\text{\raisebox{0.6 \depth}{\scalebox{0.9}[1.1]{$\chi$}}}} 
\def\vkappa{\text{{\scalebox{0.86}[1.1]{$\kappa$}}}} 
\def\vnu{\text{{\scalebox{0.9}[1]{$\nu$}}}} 
\def\uppii{\text{\scalebox{0.8}[0.96]{$\uppi$}}}

\newcommand{\BC}{{\mathbb {C}}}
\newcommand{\BQ}{{\mathbb {Q}}} 
\newcommand{\BR}{{\mathbb {R}}} 
\newcommand{\BZ}{{\mathbb {Z}}}

\newcommand{\GL}{{\mathrm {GL}}}
\newcommand{\PGL}{{\mathrm {PGL}}}
\newcommand{\SL}{{\mathrm {SL}}}

\newcommand{\ra}{\rightarrow} 
\def\viint{	\int \hskip -5 pt \int}

\def\sumx{\sideset{}{^\star}\sum}
\def\sumh{\sideset{}{^h}\sum}

\def\mod{\mathrm{mod}\,  }

\def\BSA{  \boldsymbol{\EuScript{A}} }
\def\HA{ \,\widehat{\phantom{A} } \hskip -9.5pt \BSA  }
\def\SHA{\, \widehat{\phantom{A} } \hskip -7.5pt \BSA}

\def\nd{\mathrm{d}}

\def\trh{ \mathrm{trh}}

\newcommand{\RF}{{\mathrm {F}}}

\newcommand{\overbar}[1]{\mkern 3mu\overline{\mkern-3mu#1\mkern-1mu}\mkern 1mu}

\def\SC{\text{\raisebox{- 2 \depth}{\scalebox{1.1}{$ \text{\usefont{U}{BOONDOX-calo}{m}{n}C} \hskip 0.5pt $}}}}

\def\SE{\text{\raisebox{- 2 \depth}{\scalebox{1.1}{$ \text{\usefont{U}{BOONDOX-calo}{m}{n}E} \hskip 0.5pt $}}}}

\def\lp {\left (}
\def\rp {\right )}

\def\boldJ {\boldsymbol J}
\renewcommand{\Im}{{\mathrm{Im} }}
\renewcommand{\Re}{{\mathrm{Re} }}

\def\shskip{\hskip 0.5 pt}

\def\CaloO {\text{\raisebox{- 2 \depth}{\scalebox{1.1}{$ \text{\usefont{U}{BOONDOX-calo}{m}{n}O}  $}}}}

\def\SDH{\text{\raisebox{- 6 \depth}{\scalebox{1.06}{$ \text{\usefont{U}{dutchcal}{m}{n}H}  $}}}}
\def\SDI{\text{\raisebox{- 6 \depth}{\scalebox{1.06}{$ \text{\usefont{U}{dutchcal}{m}{n} I}  $}}}}

\newcommand{\red}[1]{\textcolor{red}{#1}}

\newcommand{\delete}[1]{}

\theoremstyle{plain}

\newtheorem{thm}{Theorem} \newtheorem{cor}[thm]{Corollary}
\newtheorem{lem}[thm]{Lemma}  \newtheorem{prop}[thm]{Proposition}
\newtheorem*{conj}{Conjecture}
 
\newtheorem{manualtheoreminner}{Theorem}
\newenvironment{manualtheorem}[1]{%
	\renewcommand\themanualtheoreminner{#1}%
	\manualtheoreminner
}{\endmanualtheoreminner}

\theoremstyle{remark} 
\newtheorem{remark}{Remark}[section]

\numberwithin{equation}{section}

\begin{document}

	\title[Twisted Spectral Large Sieve  for  $\mathrm{PGL}_2 (\mathbb{Z} {[i]})   \backslash \mathrm{PGL}_2 (\mathbb{C}) $]{{On the Twisted Spectral Large Sieve Inequality for $\mathrm{PGL}_2 (\mathbb{Z}[i]) \backslash \mathrm{PGL}_2 (\mathbb{C})$}}

\begin{abstract}
	  In this paper, we extend the twisted spectral large sieve inequalities of Deshouillers, Iwaniec, Luo, and Young from $\mathrm{SL}_2 (\mathbb{Z})$ onto $\mathrm{PGL}_2 (\mathbb{Z}[i])$. 
	  
\end{abstract}

\author{Zhi Qi}
\address{School of Mathematical Sciences\\ Zhejiang University\\Hangzhou, 310027\\China}
\email{zhi.qi@zju.edu.cn}

\thanks{The author was supported by National Key R\&D Program of China No. 2022YFA1005300 and National Natural Science Foundation of China No. 12071420.}


\maketitle

\section{Introduction}

\subsection{Backgrounds}

Let $u_j (z)$ be an orthonormal basis of Hecke--Maass cusp forms for the modular group $\mathrm{SL}_2 (\BZ)$. Let $\lambda_j = s_j (1-s_j)$  and $\lambda_j (n)$ be the Laplace eigenvalue and the $n$-th Hecke eigenvalue of $ u_j (z) $. Set $s_j = 1/2+ i t_j$ ($t_j > 0$). Let $L(s, u_j)$ and $L(s, \mathrm{Sym}^2 u_j)$ be the standard and the symmetric square $L$-functions associated to $u_j (z)$. 

The non-vanishing condition of a certain $\GL_2$ Rankin--Selberg convolution with $u_j (z)$ at the  special point $s_j$ 
arises in the Phillips--Sarnak deformation theory of cusp forms \cite{Phillips-Sarnak}. 
This motivates Deshouillers, Iwaniec, and Luo \cite{DI-Nonvanishing,Luo-Twisted-LS} to establish the following large sieve inequalities for the special twisted Hecke eigenvalues $ \lambda_{j} (n) n^{it_j}  $: 
\begin{align}\label{1eq: DI's bound}
	\sum_{t_j \leqslant T} \frac 1 {L (1, \mathrm{Sym}^2 u_j)}  \bigg| \sum_{ n \leqslant N}  a_{n} \lambda_{j} (n) n^{it_j} \bigg|^2 \Lt \big(T^2 +  N^{2}\big) (TN)^{\vepsilon} \sum_{ n \leqslant N}  |a_{n}|^2,
\end{align}
\begin{align}\label{1eq: Luo's bound, 1}
	\sum_{t_j \leqslant T} \hskip -1pt   \frac 1 {L (1, \mathrm{Sym}^2 u_j)}   \bigg| \sum_{ n \leqslant N}  a_{n} \lambda_{j} (n) n^{it_j} \bigg|^2 \hskip -2pt \Lt \hskip -1pt  \big(T^2 \hskip -1pt    +  \hskip -1pt    T^{3/2} N^{1/2} \hskip -1pt   + \hskip -1pt   N^{5/4} \big) (TN)^{\vepsilon} \hskip -2pt \sum_{ n \leqslant N} \hskip -2pt |a_{n}|^2,
\end{align}
for any complex $a_n$.

The reader is referred to \cite{DI-Nonvanishing,DIPS-Maass,Luo-Non-Vanishing,Luo-Weyl,Luo-2nd-Moment} on the non-vanishing problem---it has been settled by  Luo in  \cite{Luo-Weyl}. 

The large sieve inequalities are very useful for the study of moments of $L$-functions. For instance, by \eqref{1eq: Luo's bound, 1}    Luo proved (\cite[(13), (14)]{Luo-Twisted-LS}): 
\begin{align}\label{1eq: moments, Luo}
	\sum_{t_j \leqslant T} \hskip -2pt \left|L (s_j, u_j) \right|^4 \Lt \hskip -1ptT^{2+\vepsilon}, \hskip 8pt  \sum_{t_j \leqslant T} \hskip -2pt \left|L (s_j, u_j) \right|^6 \Lt \hskip -1pt  T^{9/4+\vepsilon}, \hskip 8pt  \sum_{t_j \leqslant T} \hskip -2pt \left|L (s_j, u_j) \right|^8 \hskip -1pt \Lt T^{5/2+\vepsilon};
\end{align} 
  Luo did not write down the bound in the middle as it follows from the other two by Cauchy--Schwarz. More recently, Young  \cite{Young-GL(3)-Special-Points} and  Chandee--Li  \cite{Chandee-Li-GL(4)-Special-Points} improved Luo's bounds $ T^{9/4+\vepsilon} $ and $T^{5/2+\vepsilon}$ into the optimal mean-Lindel\"of bound  $T^{2+\vepsilon}$, and moreover, for any fixed $\SL_3 (\BZ)$ or $\SL_4 (\BZ)$ Hecke--Maass form $\phi $, they proved
\begin{align}
	 \sum_{t_j \leqslant T}   \left|L (s_j, u_j \times \phi) \right|^2 \Lt   T^{2 +\vepsilon}. 
\end{align}
Their approach utilizes a refinement of Luo's large sieve in asymptotic form, due to Young \cite[Theorem 7.1]{Young-GL(3)-Special-Points},  which is conducive to further analysis using the   Vorono\"i summation for the Fourier coefficients of $\phi$. 

The main purpose of this paper is to extend the twisted large sieve inequalities of Deshouillers, Iwaniec, and Luo, as well as the refinement of Young from the modular group $\SL_2 (\BZ)$ to the Picard group $\PGL_2 (\BZ [i])$.  It should be remarked that an extra weaker error term arises in the complex context, which is satisfactory in the $\GL_2 \times \GL_3$ but {\it not} in the  $\GL_2 \times \GL_4$ setting. 

\subsection{Notation} \label{sec: notation}

Denote $\RF = \BQ (i)$ and $ \CaloO  = \BZ [i]$. Let $\CaloO^{\times}$ be the group of units and $\CaloO    / \CaloO^{\times}$ denote the set of   ideals, usually identified with the Gaussian integers in the first quadrant. 
The letters $m$, $n$, $c$,  and $q$ will be reserved for  (non-zero)  Gaussian integers in $\CaloO    $. Let $(n)$ denote the ideal $n \CaloO$. 

The letters $u$, $v$, $w$, and $z$ will be reserved for complex variables in $\BC$. Let $\nd z$ be the   Lebesgue measure on $\BC$. 

Let $e [z] = \exp (2\pi i\Re (z))$.  For $  m,  n \in \CaloO $ and $c, q \in \CaloO \smallsetminus \{0\}$,      define 
\begin{align}\label{1eq: defn Kloosterman}
	S   (m, n ; c ) = \sumx_{   \valpha  (\mathrm{mod} \, c) } e \bigg[  \frac {  \valpha m +   \widebar{\valpha} n } {c} \bigg] ,
\end{align} 
\begin{align} \label{1eq: defn V}
		V_{q} (m, n; c) =	\mathop{\sum_{\valpha (\mathrm{mod}\, c)}}_{ (\valpha (q-\valpha), c) = (1) }  e \bigg[  \frac {  \widebar{\valpha} m +  \overline{q - \valpha} n } {c} \bigg] ,
\end{align}
where the $\star$ indicates the condition $(\valpha, c) = (1)$ and $ \widebar{\valpha}$ is given by $\valpha \widebar{\valpha}  \equiv 1 (\mathrm{mod} \, c)$. 
 


The   (unitary) characters on $\BC^{\times} $ are of the form     \begin{align}
	\vchi_{i \vkappa,  p} (z) = |z|^{  i \vkappa  } (z/|z|)^{  p}  , 
\end{align}  for   $ \vkappa   $ real and   $ p$ integral, 
so we introduce  $\BSA     = \BR \times  \BZ$ to parameterize the (Mellin) unitary dual of $\BC^{\times}  $.  
Define $ \HA = \BR \times \BR / \pi \BZ$, which will be considered as the (Fourier) dual of $  \BSA $. 
Let $\nd \mu (\vkappa,  p) $ or $\nd \widehat{\mu} (r, \omega)$  denote the usual Lebesgue measure on  $\BSA    $ or $ \HA$ respectively.  For simplicity, write $ \vchi_{  p} (z) = \vchi_{0,  p} (z)$.

For $p \in \BZ $,  define 
\begin{align}\label{1eq: zeta}
	\zeta (s, p) =  \sum_{(n) \shskip \subset \CaloO } \frac {\vchi_{4p} (n) } {|n|^{2s}}, \qquad \text{($\Re (s) > 1$)}, 
\end{align}
to be the Hecke $\zeta$ function associated to the Gr\"ossencharakter $\vchi_{4p} : (n) \ra (n/|n|)^{4 p}$, and   define 
\begin{align}\label{1eq: tau s}
	\tau_{s, p} (n) = \sum_{ (a) (b) = (n)} \vchi_{4p} (a/b) |a/b|^{2s} .
\end{align}  
Note that $\tau_{- s, - p} (n) = \tau_{s, p} (n)$. 

For $(\vkappa, p) \in \BSA$, let $\uppii_{i \vkappa, p} $ be the unitary principle series of $\mathrm{PGL}_2 (\BC)$: the unique infinite dimensional constituent of the representation   unitarily induced from the character
\begin{align*}
\vchi_{i \vkappa,  p}: 	\begin{pmatrix}
		\sqrt{z} & v \\
		& 1/\sqrt{z}
	\end{pmatrix}   \ra 
	|z|^{i \vkappa} (z/|z|)^{  p}. 
\end{align*}
Note that  $\uppii_{-i \vkappa, -p} \approx   \uppii_{i \vkappa, p} $.
As the Selberg conjecture holds for $\mathrm{PGL}_2 (\CaloO)$, we do not consider complementary series. 

Let  $  L^2_{c} (\mathrm{PGL}_2 (\CaloO) \backslash \mathrm{PGL}_2 (\BC))  $ be the space of  cusp forms for $ \mathrm{PGL}_2 (\CaloO) $.  Let $\Pi_{c} $ denote the discrete   spectrum of the irreducible constituents of $L^2_{c} (\mathrm{PGL}_2 (\CaloO) \backslash \mathrm{PGL}_2 (\BC))$. 

It may be assumed that each  $\uppii \in \Pi_{c}$ is Hecke invariant. Let $\lambda_{\uppii} (n)$ be the Hecke eigenvalues of $ \uppii $. It is known that $\lambda_{\uppii} (n)$ are all real.  Note that the $\mathrm{PGL}_2 (\CaloO) $-invariance ensures that $ \lambda_{\uppii} (n) $ only depends on the ideal $(n)$.

For $\uppii \in \Pi_{c}$, let $  (\vkappa_{\uppii}, p_{\uppii})$ be the parameter $(\vkappa, p) \in \BSA$ such that  $ \uppii \simeq \uppii_{i \vkappa , p }$. 

Let $L(s, \uppii)$ and $L (s, \mathrm{Sym}^2 \uppii)$ be the standard and the symmetric square $L$-functions attached to $\uppii$.   

In this paper, we shall mainly consider those $ \uppii $ with $   p_{\uppii} \equiv 0 \, (\mathrm{mod}\, 4)$. Let $ \Pi_{c}^{\natural} $ be the set of such $\uppii$. 
Define 
\begin{align*}
\Pi_{c}^{\natural}  (K, P) =  \big\{ \uppii \in \Pi_{c}^{\natural}  : |\vkappa_{\uppii} | \leqslant K,   |p_{\uppii} | \leqslant P \big\}, \qquad 
	\Pi_{c}^{\natural}  (T) =  \big\{ \uppii \in \Pi_{c}^{\natural}  : |\vkappa_{\uppii} |,   |p_{\uppii} | \leqslant T \big\}.   
\end{align*}
 
 As a convention, throughout the paper,  $\vepsilon  $ is arbitrarily small and its value  may differ from one occurrence to another.



\subsection{Main Results}


The main results of this paper are the analogs   over $\RF = \BQ (i)$  of the twisted spectral large sieve inequalities of Deshouillers, Iwaniec, and Luo. 

\begin{thm}\label{thm: prelim} 
Let $K, P , N  \geqslant  1  $  and $\vepsilon > 0$ be real numbers.  Let $\boldsymbol{a} : \CaloO / \CaloO^{\times}   \ra \BC$ be a sequence. 
Define 
\begin{align}
	\label{1eq: sum, cusp}	
	\SC  (K, P, N) =	  \sum_{\uppii \shskip \in \Pi_{c}^{\natural}  (K, P) }  \frac { 1  } { L (1, \mathrm{Sym}^2 \uppii) } \bigg| \sum_{     |n|  \leqslant    { N} }  a_{(n)} \lambda_{\uppii} (n) \vchi_{i \vkappa_{\uppii},  p_{\uppii} } (n) \bigg|^2 . 
\end{align}	 Then
	\begin{align}\label{1eq: bound, weaker}
		\SC  (K, P, N) \Lt 	\big(K^2 + P^2\big) \big(  K P +   { N^{2 }}    \big) N^{\vepsilon} \sum_{    |n|  \leqslant    {N} } |a_{(n)}|^2  . 
	\end{align}
\end{thm}

\begin{thm}\label{thm: large sieve, 2} 
	Let $T , N  \geqslant  1 $  and $\vepsilon > 0$.  Let $\boldsymbol{a} : \CaloO / \CaloO^{\times}   \ra \BC$. 
	Define 
	\begin{align}
		\label{1eq: sum, cusp, 2}	
		\SC  (T, N) =	  \sum_{\uppii \shskip \in \Pi_{c}^{\natural}  (T) }  \frac { 1  } { L (1, \mathrm{Sym}^2 \uppii) } \bigg| \sum_{     |n|  \leqslant    { N} }  a_{(n)} \lambda_{\uppii} (n) \vchi_{i \vkappa_{\uppii},  p_{\uppii} } (n) \bigg|^2 . 
	\end{align}	 Then for $ N > T$ we have 
	\begin{align}\label{1eq: bound, large sieve, 2}
		\SC  (T, N) \Lt \big(  T^3 N  +  T N^2  +  N^{5/2} \big) N^{\vepsilon} \sum_{    |n|  \leqslant    {N} } |a_{(n)}|^2  . 
	\end{align}
\end{thm}

\begin{manualtheorem}{2$'$}\label{thm: large sieve}
Let the notation be as in Theorem {\rm\ref{thm: prelim}}. 
We have the bound
	\begin{equation}\label{1eq: bound, large sieve}
	\begin{split}
		\SC  (K, P, N) \Lt   	\big(   \big(K^2 + P^2\big)  K P   +   (K+P)^3 N  +  (K+P) N^2  +  N^{5/2} \big) N^{\vepsilon} \sum_{    |n|  \leqslant    {N} } |a_{(n)}|^2 .
	\end{split}
	\end{equation} 
\end{manualtheorem}

In the case $N > K+P$,  Theorem \ref{thm: large sieve} follows from Theorem \ref{thm: large sieve, 2} by the embedding $ \Pi_{c}^{\natural}  (K, P) \hookrightarrow \Pi_{c}^{\natural}  (K + P) $. In the case $N \leqslant  K+P$, Theorem \ref{thm: large sieve}  follows directly  from Theorem \ref{thm: prelim}.

Note that the bound in \eqref{1eq: bound, weaker}  is better in the $N$-aspect than Deshouillers and Iwaniec's  as in \eqref{1eq: DI's bound}, while the bound 
in \eqref{1eq: bound, large sieve}  is relatively weaker than Luo's  as in \eqref{1eq: Luo's bound, 1}.  
The term $T N^2$ in \eqref{1eq: bound, large sieve, 2}, which corresponds to the $(K+P) N^2$ in \eqref{1eq: bound, large sieve},  arises in the approximation of certain trigonometric-hyperbolic functions. It is understood that  if one changes the setting  from $\BR$ to $\BC$, then the degree (of everything) needs to be doubled, 
but unfortunately it is not the case for our approximation.   For the details see Remark \ref{rem: Luo}. 


 Theorem \ref{thm: prelim} will be derived right after the application of Kuznetsov in \S \ref{sec: proof prelim}, while Theorem \ref{thm: large sieve, 2} will require more delicate analysis: simply speaking,  there is cancellation between the continuous spectrum and the Kloosterman part as observed by   Luo \cite{Luo-Twisted-LS}. However, instead of the Euler--Maclaurin formula used by   Luo  \cite{Luo-Twisted-LS}, we follow   Iwaniec,   Li, and  Young \cite{Iwaniec-Li-Ortho,Young-GL(3)-Special-Points} and apply   Poisson summation  after the application of Kuznetsov and a certain approximation of the Bessel integral.  Moreover, we remark that our analysis is facilitated by the choice of Gaussian test functions and their self-dual property under the Fourier transform.
 
 The reader might wonder: 
 \begin{itemize}
 	\item [(1)] Why do we work on  $\Pi_{c}^{\natural} (T) $ as in Theorem \ref{thm: large sieve, 2}? 
 	\item [(2)] Why do we use the embedding $ \Pi_{c}^{\natural}  (K, P) \hookrightarrow \Pi_{c}^{\natural}  (K + P) $ rather than work on $ \Pi_{c}^{\natural}  (K, P)  $ directly in the proof of Theorem \ref{thm: large sieve}? 
 \end{itemize}   Of course, the case of  $\Pi_{c}^{\natural} (T) $ is considered to be simpler, but the real reason is that in the case of $ \Pi_{c}^{\natural}  (K, P)  $ there would always be a loss of power of $K/P+P/K$ at each step of the analysis, and the accumulation of such losses will eventually yields a worse bound,  so one has {\it no} benefit at all to work on $ \Pi_{c}^{\natural}  (K, P)  $ instead of $\Pi_{c}^{\natural} (T) $. 
 The reader is referred to \cite[\S 1.5]{Qi-Spectral-LS} for the discussion of a simpler example.

\subsection{Bounds for the Moments of $L$-functions with Special Twists}  
For $\uppii \in \Pi_{c}^{\natural}$ consider the  $L$-function of $\uppii$ with the special twist $ \vchi_{i \vkappa_{\uppii},  p_{\uppii} } $: 
\begin{align*}
	L (s, \uppii \times \vchi_{i \vkappa_{\uppii},  p_{\uppii} }) = \sum_{(n) \subset \CaloO} \frac { \lambda_{\uppii} (n) \vchi_{i \vkappa_{\uppii},  p_{\uppii} } (n) } {|n|^{2 s}}, \qquad \text{($\Re (s) > 1$)}; 
\end{align*}
of course it may also be  written as $L (s- i\vkappa_{\uppii}/2, \uppii \times \vchi_{ p_{\uppii} })$.  The twisted $L$-function admits analytic continuation to the whole $s$-plane and satisfies the functional equation 
\begin{align*}
	 \Lambda (s , \uppii \times \vchi_{i \vkappa_{\uppii},  p_{\uppii} }) =   \Lambda (1-s , \uppii \times \overline{\vchi}_{i \vkappa_{\uppii},  p_{\uppii} }), 
\end{align*}
where 
\begin{align*}
	\Lambda (s , \uppii \times \vchi_{i \vkappa_{\uppii},  p_{\uppii} }) = \pi^{-2 s} \Gamma (s) \Gamma  ( s + i\vkappa_{\uppii} + | p_{\uppii}|  ) L (s, \uppii \times \vchi_{i \vkappa_{\uppii},  p_{\uppii} }) . 
\end{align*} It is a special feature that the analytic conductor is dropped from $ |   i \vkappa_{\uppii} + p_{\uppii} |^4 $ down to $ |   i \vkappa_{\uppii} + p_{\uppii} |^2 $ due to the twist by $\vchi_{i \vkappa_{\uppii},  p_{\uppii} }$, so that $L (s , \uppii \times \vchi_{i \vkappa_{\uppii},  p_{\uppii} })$ behaves   like  the Hecke $\zeta$ function  for a Gr\"ossencharakter, although the coefficients are genuine $\GL_2$ objects. Consequently, for $ \uppii \in \Pi_{c}^{\natural} (K, P)$ one may approximate $ L (1/2 , \uppii \times \vchi_{i \vkappa_{\uppii},  p_{\uppii} }) $ very well by its partial sum over $|n| \leqslant (K+P)^{1/2+\vepsilon} $. 

One may prove 
\begin{equation}
	  \sum_{\uppii \shskip \in \Pi_{c}^{\natural}  (K, P) } \big| L (1/2, \uppii \times \vchi_{i \vkappa_{\uppii},  p_{\uppii} })  \big|^2 \Lt KP \big(K^2+P^2\big)^{1+\vepsilon},  
\end{equation}
by Theorem \ref{thm: prelim}, and, similar to   Luo's moment bounds in \eqref{1eq: moments, Luo}, 
\begin{equation}
   \sum_{\uppii \shskip \in \Pi_{c}^{\natural}  (K, P) } \big| L (1/2, \uppii \times \vchi_{i \vkappa_{\uppii},  p_{\uppii} })  \big|^4 \Lt (K+P)^{4+\vepsilon}, 
\end{equation}
\begin{align}
	   \sum_{\uppii \shskip \in \Pi_{c}^{\natural}  (K, P) } \big| L (1/2, \uppii \times \vchi_{i \vkappa_{\uppii},  p_{\uppii} })  \big|^6 \Lt   (K+P)^{9/2+\vepsilon}, 
\end{align}
\begin{align} 
	  \sum_{\uppii \shskip \in \Pi_{c}^{\natural}  (K, P) } \big| L (1/2, \uppii \times \vchi_{i \vkappa_{\uppii},  p_{\uppii} })  \big|^8 \Lt   (K+P)^{5+\vepsilon}, 
\end{align}
by Theorem  \ref{thm: large sieve}. In view of the embedding  $ \Pi_{c}^{\natural}  (K, P) \hookrightarrow \Pi_{c}^{\natural}  (K + P) $ in the argument above, the most optimistic bounds by the method are   $  O \big( (K+P)^{4+\vepsilon}\big)$.  Thus the mean-Lindel\"of bounds over $  \Pi_{c}^{\natural}  (T)$ are the best results we can hope for at present. 

\begin{conj}  The following mean-Lindel\"of bounds should hold{\rm:}
	\begin{align} 
		\sum_{\uppii \shskip \in \Pi_{c}^{\natural}  (T) } \big| L (1/2, \uppii \times \vchi_{i \vkappa_{\uppii},  p_{\uppii} })  \big|^6 \Lt   T^{4+\vepsilon}, \quad \sum_{\uppii \shskip \in \Pi_{c}^{\natural}  (T) } \big| L (1/2, \uppii \times \vchi_{i \vkappa_{\uppii},  p_{\uppii} })  \big|^8 \Lt   T^{4+\vepsilon},
	\end{align}
and  
\begin{align} 
	\sum_{\uppii \shskip \in \Pi_{c}^{\natural}  (T) } \big| L (1/2, \uppii \times \phi \times \vchi_{i \vkappa_{\uppii},  p_{\uppii} })  \big|^2 \Lt   T^{4+\vepsilon},
\end{align}
if   $\phi$ is a given cuspidal automorphic representation  for $\GL_3 (\CaloO)$ or $\GL_4 (\CaloO)$.
\end{conj}

\subsection{Setup} 

Let    
\begin{equation}\label{1eq: weight k} 
	h ( \vkappa,  p ) = \exp \left(  - \Big( \frac {\vkappa } {K} \Big)^2 - \Big( \frac {p } {P } \Big)^2 \right) .
\end{equation}
Assume as we may that $a_{(n)}$ is a real sequence and that $ a_{(n)} = 0$   unless $  {N} <   |n|  \leqslant    {2N}  $. Define
\begin{align*}
	\| \boldsymbol{a}_N \|_2^{}   = \bigg(  \sum_{  {N} <  |n|  \leqslant    {2N} } |a_{(n)}|^2 \bigg)^{1/2}. 
\end{align*}
Subsequently, we shall deal with  the smoothed
\begin{align}
	\label{2eq: sum, cusp}	
	\SC  (\boldsymbol{a}) =	   \sum_{\uppii \shskip \in \Pi_{c}^{\natural} }  \frac { h  ( \vkappa_{\uppii},  p_{\uppii} ) } { L (1, \mathrm{Sym}^2 \uppii) } \bigg| \sum_{ n  }  a_{(n)} \lambda_{\uppii} (n) \vchi_{i \vkappa_{\uppii},  p_{\uppii} } (n) \bigg|^2 ,
\end{align}
and 
\begin{align}	\label{2eq: sum, Eis}	
	\SE (\boldsymbol{a}) =	 \frac 1 {\pi} 	\viint_{ \BSA } \frac {h (2\vkappa, 4 p)} { | \zeta (1+2 i \vkappa, 2 p)   |^2  } \bigg| \sum_{ n }  a_{(n)} \tau_{i \vkappa, p} (n) \vchi_{ 2 i \vkappa, 4 p} (n)  \bigg|^2  \nd \mu (\vkappa, p ) . 
\end{align}
Note that one may as well change $ \uppii   \in \Pi_{c}^{\natural}$ into $\uppii   \in \Pi_{c} $ in \eqref{2eq: sum, cusp}, as the $n$-sum vanishes if $\vchi_{i \vkappa_{\uppii},  p_{\uppii} }   (n)$ is not a Gr\"ossencharakter. This simple observation will be needed when we apply the Kuznetsov trace formula. 

\subsection{Strategy for the Proof of Theorem \ref{thm: large sieve, 2}} 
In the rest of Introduction, set $K = P = T$.  
Define the bilinear form 
\begin{align}\label{2eq: Sigma (a)}
	\varSigma (\boldsymbol{a}) =  T^2  \mathop{\sum \sum}_{m, n} a_{(m)} a_{(n)} \sigma (m ,n) , \qquad \sigma (m, n) = \sum_{c }  \frac {S (m, 0; c) S (n, 0; c)} {|c|^4}. 
\end{align}

The idea for the proof of Theorem \ref{thm: large sieve, 2} is to extract the same main term---a multiple of $ \varSigma (\boldsymbol{a})     $---from the Eisenstein series contribution $\SE (\boldsymbol{a})$ and the Kloosterman part in $\SC  (\boldsymbol{a}) + \SE (\boldsymbol{a})$ after the Kuznetsov formula, and then do the cancellation.  More precisely, we shall prove the following asymptotic formulae.

\begin{prop}\label{prop: E(a)}
Let $ N >  T $. Then
	\begin{align}\label{2eq: asymptotic E(a)}
		\SE  (\boldsymbol{a}) =	 \frac {1} {32 } \varSigma (\boldsymbol{a}) +  O \big(  N^{2+\vepsilon} \|\boldsymbol{a}_{N} \|_2^2 \big) . 
	\end{align}
\end{prop}

\begin{prop}\label{prop: C(a)}
 Let $ N >  T $. Then
 	\begin{align} \label{1eq: C(a)+E(a)}
			\SC  (\boldsymbol{a}) + \SE (\boldsymbol{a})  =	 \frac {1} {32 } \varSigma (\boldsymbol{a}) + O \bigg(   \bigg(  T^3 \min \big\{ N,   T^3 \big\} +   T N^2 + \frac { N^3} {T^2}   \bigg) N^{\vepsilon} \|\boldsymbol{a}_{N} \|_2^2 \bigg).   
	\end{align} 
\end{prop}

It follows from Propositions  \ref{prop: E(a)} and \ref{prop: C(a)}   that  
\begin{align}
	\SC (T, N) \Lt	 \bigg(  T^3 \min \big\{ N,  T^3 \big\} +  T N^2 + \frac { N^3} {T^2}   \bigg) N^{\vepsilon}  \sum_{    |n|  \leqslant    {N} } |a_{(n)}|^2 ,
\end{align}
as long as $  N > T $. 
Since $ \SC (T, N) $ is non-decreasing in $T $, on replacing $T$ by $T + \sqrt{N }$ we arrive at  \eqref{1eq: bound, large sieve, 2} in Theorem \ref{thm: large sieve, 2}. 

\subsection{Refined Large Sieve}

Finally, we shall also establish (see in particular Proposition \ref{prop: S})  the analog  of  Young's refined large sieve \cite[Theorem 7.1]{Young-GL(3)-Special-Points}. 

\begin{thm}\label{prop: C(a), 2}
	Let $ 1 \leqslant X \leqslant T < N  $. Then
	\begin{align} \label{1eq: C(a)+E(a), 2}
		\SC  (\boldsymbol{a})   =    S_0 (\boldsymbol{a}; X) + O \bigg(   \bigg( T N^2 + \frac { N^3} {T^2}  + \frac {T^2 N^2} {X^2}    \bigg) N^{\vepsilon} \|\boldsymbol{a}_{N} \|_2^2 \bigg), 
	\end{align} 
	with  
	\begin{equation}\label{1eq: S0(a;X)}
		\begin{split}
			S_0 (\boldsymbol{a}; X) \Lt   T^4 N^{\vepsilon} \hskip -1pt  \sum_{|c| \leqslant X}  \frac 1 {|c|^2}  	\sum_{0 < |q| \leqslant |c|   N^{\vepsilon}     } \hskip -1pt   \frac 1 {|q|^2} \hskip -1pt    \viint_{\overbar{D}_{  N^{\vepsilon}/T }}  \hskip -1pt   \bigg|  \sum_{ n}    {a}_{(n)}  S (n, q; c) e \Big[ \frac {n} {c} u  \Big] \bigg|^2 \hskip -1pt \nd u ,
		\end{split}
	\end{equation}
	where $D_{\rho}$ is    the disc  of radius $\rho$ centered at the origin. 
\end{thm} 

Note that the error terms in \eqref{1eq: C(a)+E(a), 2} are all $O \big( T^{4+\vepsilon} \|\boldsymbol{a}_{N} \|_2^2 \big)$ as long as $N \leqslant T^{3/2+\vepsilon}$ if we choose $X \approx T^{1/2 }$. Therefore, given the asymptotic formula for $\GL_3 (\BC)$ Bessel functions in \cite[Theorem 16.6]{Qi-Bessel} (indeed for $\GL_n(\BC)$),  it seems possible to extend the $\GL_2 \times \GL_3$ mean-Lindel\"of bound of Young  \cite[Theorem 3.1]{Young-GL(3)-Special-Points} by applying  the  Vorono\"i summation for $\GL_3 (\CaloO)$. However, the term $T N^{2}$ exceeds $T^{4+\vepsilon}$ by $T$ for $N = T^{2+\vepsilon}$, so it is not yet possible to extend the $ \GL_2 \times \GL_4 $ result of   Chandee and   Li \cite{Chandee-Li-GL(4)-Special-Points}. 



\section{\texorpdfstring{Spectral Kuznetsov Trace Formula for $\mathrm{PGL}_2 (\BZ [i])$}{Spectral Kuznetsov Trace Formula for PGL\unichar{"2082}(Z[i])}} \label{sec: Kuznetsov}

Our main tool  is the spectral Kuznetsov trace formula of Bruggeman and Motohashi over the Gaussian field $\mathrm{F}$ \cite[Theorem 10.1]{B-Mo}. The version here in our set of notation is from \cite[Lemma 5]{Qi-Spectral-LS}.


For $\vnu \in \BC$, $  p \in \BZ  $, and $z \in \BC \smallsetminus \{0\}$,  define 
\begin{equation}\label{0def: J mu m (z)}
	J_{\vnu ,   p} (z) = J_{\vnu + p }    (z)   J_{\vnu -  p  }    ({\widebar z} ) ,
\end{equation} 
\begin{equation}\label{0eq: defn of Bessel}
	\boldJ_{ \vnu,   p} (z)  =   \frac {2\pi^2} {\sin (\pi \vnu)}  ( J_{-\vnu,\shskip  -p} (   z) - J_{\vnu,   p} (    z)  ) ,   
\end{equation}
where $ J_{\nu} (z) $ is the Bessel function of the first kind  (see \cite{Watson}).  



\begin{lem}\label{lem: Kuznetsov}
Let the notation be as in {\rm \S \ref{sec: notation}}.  Let $h (\vkappa, p)$ be an even function on $\BSA$   that admits an entire analytic continuation $ h (\vkappa + i \sigma, p) $ so that it  decays rapidly in both $ \vkappa$ and $ p $,  uniformly for $\sigma$ on   bounded intervals.  
	For $m , n \in \CaloO \smallsetminus \{0\}$, we have the identity{\rm:} 
	\begin{equation}\label{1eq: Kuznetsov} 
		\begin{aligned}
		  \sum_{\uppii \shskip \in \Pi_{c} }  &    \frac {  \lambda_{\uppii} ( m )     \lambda_{\uppii}  ( n )} {L(1, \mathrm{Sym^2} \uppii )}   h  ( \vkappa_{\uppii}, p_{\uppii} ) +   \frac 1 {\pi}   \viint_{\BSA}  
			\frac {\tau_{i \vkappa, p} (m )  \tau_{i \vkappa, p} ( n )} {|\zeta (1+2i \vkappa, 2p)|^2} h ( 2\vkappa, 4 p )  \shskip   \nd  \mu (\vkappa, p)  \\
			& =        \frac {1} { 16\pi^3 }  \delta_{(m),   (n)} \cdot  \SDH  + \frac 1 {64 \pi^3 }  \sum_{ \epsilon \shskip  \in   \CaloO^{\times} \hskip -1pt / \CaloO^{\times 2} } \sum_{c  \shskip \in   \CaloO \smallsetminus \{0\} } \frac {S  (  m , \epsilon  n  ; c  ) } { |c|^2  } \SDH  \bigg( \frac { 2\pi \sqrt{\epsilon m n} } {    c     }   \bigg),
		\end{aligned}
	\end{equation}
	where  $\delta_{(m),   (n)}$ is the Kronecker $\delta$ symbol for ideals, $ \SDH  $ and $ \SDH (z)$ are the Plancherel and Bessel integrals defined by  
	\begin{align}\label{1eq: defn Bessel integral}
		\SDH  = \hskip -2pt  \viint_{\BSA}   h (\vkappa, p)  (\vkappa^2 + p^2 )  \nd  \mu (\vkappa, p), \quad \SDH (z) = \hskip -2pt  \viint_{\BSA}   h (\vkappa, p)  \boldsymbol{J}_{  i \vkappa, p} ( z )  (\vkappa^2 + p^2 )  \nd  \mu (\vkappa, p)  . 
	\end{align} 
\end{lem}


\section{Preliminary Analysis for the Bessel Integral}


For $u \in \BC \smallsetminus \{0\}$, let    
\begin{equation}\label{4eq: choice of h} 
	h( \vkappa,  p; u ) =  h (\vkappa, p)  	\cos  (2 \vkappa \log |u| + 2 p \arg (u)  ),
\end{equation}
where $ h (\vkappa, p)  $ is defined as in \eqref{1eq: weight k} (not in \S \ref{sec: Kuznetsov}).
In this section, we investigate the associated Bessel integral
\begin{equation}\label{5eq: Bessel H(z)}
	\SDH (z; u) = \hskip -2pt  \viint_{\BSA}   h (\vkappa, p; u)  \boldsymbol{J}_{  i \vkappa, p} ( z )  (\vkappa^2 + p^2 )  \nd  \mu (\vkappa, p)  .
\end{equation}
To this end, we follow closely the line of arguments in \S 4 of \cite{Qi-Spectral-LS}. Note that $ h( \vkappa,  p; 1 ) =  h (\vkappa, p) $, so Lemmas \ref{lem: integral repn} and \ref{lem: integral repn, 2} below may be viewed as a generalization of Lemmas 5 and 6 in \cite{Qi-Spectral-LS}.  However, unlike \cite{Qi-Spectral-LS}, the formulae in Lemmas  \ref{lem: integral repn} and \ref{lem: integral repn, 2} are too complicated and not directly applicable for our purpose---further analysis will be conducted in  the next section \S \ref{sec: further analysis} (in the square case $K=P=T$).

\begin{lem}\label{lem: integral repn}
We have the formula 
	\begin{align}\label{4eq: integral H(z)}
		\SDH (z; u) =   \viint_{ {\SHA}  }  \cos  ( \Re  ( v       \trh  (r, \omega) - w       \trh'  (r, \omega) ) )  f  (r, \omega)    \nd \widehat{\mu} (r,  \omega ) , 
	\end{align}
	where 
	\begin{align}\label{4eq: w+-}
		v       = z u + z/u, \qquad   w       = zu - z/u, 
	\end{align}
	\begin{equation}\label{4eq: trh+-}
		\begin{split}
			&	\trh (r, \omega)    =       \cosh r \cos \omega + i \sinh r \sin \omega,  \quad \hskip -1pt    \trh'  (r, \omega)    =        \sinh r \cos \omega + i  \cosh r \sin \omega, 
		\end{split}
	\end{equation} 
	\begin{align}\label{4eq: f (r, w)}
		f  (r, \omega) =	 - k ''(r) \theta  (\omega) - k  (r) \theta '' (\omega) , 
	\end{align}
	\begin{align}\label{4eq: defn of h, theta}
		k   (r )   =  \sqrt{\pi} K \exp \big( \hskip -1pt    -  (K r )^2   \big), \qquad 	
		\theta  (\omega) =   \sqrt{\pi} P \sum  \exp \big( \hskip -1pt   - (P (\omega + \pi p ) )^2 \big)    . 
	\end{align} 
\end{lem}

\begin{lem}\label{lem: integral repn, 2}
	
	We have the formula 
	\begin{align}\label{4eq: integral H(z), 2}
		\SDH (z; u)    =       \viint_{ {\SHA}  }    \cos  ( \Re  ( v       \trh  (r, \omega) - w       \trh'  (r, \omega) ) )   g ( r, \omega; v      ,  w      ) k  (  r) \theta  (  \omega)    \nd \widehat{\mu} (r,  \omega ) ,
	\end{align} 
	where $g ( r, \omega; v      ,  w      )$ is defined to be
	\begin{equation}\label{4eq: f (w; r w)}
		\begin{split}
		  (\sinh^2 r      +       \sin^2 \omega  ) 	  | v      |^2    +   (\cosh^2 r      -      \sin^2 \omega  )  | w      |^2  
		& - \frac 1 2 \Re  (( \sinh 2 r + i \sin 2 \omega ) v \overbar{w} )       . 
			\end{split}
	\end{equation}  
\end{lem}

Later, it will be more convenient to consider the variant Bessel integral
\begin{align}\label{4eq: I(v;w), 1}
	\SDI (v, w) = \viint_{ {\SHA}  } 
	\exp (i \Re (v   ( \trh  (r, \omega) -1))) \cos  ( \Re   ( w  \trh'  (r, \omega)   )   ) f  (r, \omega)   \nd \widehat{\mu} (r,  \omega ),
\end{align}
or, in the second form, 
\begin{equation}\label{4eq: I(v;w), 2}
   \begin{split} 
   	\viint_{ {\SHA}  }   \hskip -1pt \exp (i \Re (v   ( \trh  (r, \omega) \hskip -1pt    -    \hskip -1pt    1)  \hskip -0.5pt  )  \hskip -0.5pt    ) \cos  ( \Re   ( w  \trh'  (r, \omega) \hskip -0.5pt   ) \hskip -0.5pt   )    g ( r, \omega; v      ,  w      ) k  (  r) \theta  (  \omega)    \nd \widehat{\mu} (r,  \omega ).
   \end{split}
\end{equation}

In practice, for $m, n, c \in \CaloO \smallsetminus \{0\}$,   
\begin{align}
	z = \frac {2\pi \sqrt{m n}} {c}, \qquad u = \sqrt{\frac {m} {n} },   
\end{align}
as in \eqref{6eq: P(a)}, and hence
\begin{align}\label{4eq: v, w = m+-n}
	v    = 2\pi \frac {m + n} {c}, \qquad w       = 2\pi \frac {m - n} {c} ,
\end{align}
as in \eqref{6eq: P(a), 2}. 

A  pleasant  feature is that $m$ and $n$ become separate inside the integrals. 
The first formula \eqref{4eq: integral H(z)} or its variant \eqref{4eq: I(v;w), 1} is of a simpler form and will be used for the analysis in \S \ref{sec: further analysis}. 
The second formula \eqref{4eq: integral H(z), 2} or \eqref{4eq: I(v;w), 2}, in particular,  the factors $|v     |^2$, $|w     |^2$, $v \overbar{w}$ as in \eqref{4eq: f (w; r w)},    
will be used mainly  to ensure the decay  of the $c$-sum in Kuznetsov.  


\begin{proof}[Proof of Lemma \ref{lem: integral repn}] 
	From    \cite[(4.7)]{Qi-Spectral-LS},  we invoke the integral representation:
	\begin{equation}\label{10eq: Bessel}
		\boldJ_{i \vkappa ,   p}  (  x     e^{ i \phi} )   = 4 \pi i^{2p}  \int_{-\infty}^\infty   \widebar{\vchi}_{2p}  (\cosh (r + i\phi) )   J_{2p} \hskip -1pt \left( 2  {x} |\cosh (r + i\phi)|  \right) \exp ({2 i r \vkappa})  \nd r ,
	\end{equation}   
	where $\vchi_{2p} (z) = (z/|z|)^{2p}$ as in \S \ref{sec: notation}. 
	For  $|r| > 1$,  we have the uniform bound 
	 \begin{align*}
	 	J_{2p} \hskip -1pt \left( 2   {x} |\cosh (r + i\phi)| \right) \Lt \frac {\sqrt{1+p^2}} {\sqrt{x \cosh r}  } . 
	 \end{align*}
	 Thus the triple integral obtained by inserting \eqref{10eq: Bessel} into \eqref{5eq: Bessel H(z)} is absolutely convergent, and  hence it is permissible  to change the order of integration.

	 Now let the $r$-integral as in \eqref{10eq: Bessel} be the outermost, and apply the following variant of the Bessel formula  (as in \cite[(4.8)]{Qi-Spectral-LS}):
	 \begin{align}\label{10eq: Bessel, 2}
	 	\widebar{\vchi}_{2p} (a) J_{2p} (  |a|) = \frac 1 {2 \pi i^{2p}	} \int_0^{2\pi} \exp ( {2 i p \omega + i   \mathrm{Re} ( a e^{i \omega}  )} ) \nd \omega ,
	 \end{align}
  with  $ a = 2 x \cosh (r + i\phi)$ so that $ \mathrm{Re} ( a e^{i \omega}  ) = 2 x \shskip \trh (r, \omega; \phi)$, where
	 \begin{align*}
	 	\trh  (r, \omega; \phi) =    \cosh r \cos \omega \cos \phi - \sinh r \sin \omega \sin \phi .
	 \end{align*}     It follows that the Bessel integral $\SDH	(  x     e^{ i \phi}; y e^{i \theta} )$ is equal to
	 \begin{align*}
	 	4  &  \viint_{ {\SHA}  }  \cos ( 2 x \trh (r, \omega; \phi)) \nd \widehat{\mu} (r,  \omega )\\
	 	\cdot &  \viint_{\BSA    }     \cos  ( { 2   \vkappa (r+\log  {y}) +  2   p (\omega +\theta )}  ) k (\vkappa, p)     (\vkappa^2 + p^2 )  \nd  \mu (\vkappa, p)     , 
	 \end{align*}
	 in the notation of \S \ref{sec: notation} ($\BSA = \BR \times \BZ$ and $\HA = \BR \times \BR /\pi \BZ$). 

	 Next, we make the change of variables $   r \ra  r - \log y   $ and $  \omega   \ra   \omega - \theta$. By a simple calculation, we find that $2 x \,	\trh  (r - \log  {y}, \omega - \theta  ; \phi)$ is equal to 
	 \begin{align*}  
	 	xy \cdot (\cosh r - \sinh r)  	 \cos (\omega - \theta - \phi)   + 	x/y \cdot (\cosh r + \sinh r)  	 \cos (\omega - \theta + \phi).  
	 \end{align*}
	 Further, for $z = x e^{i\phi}$ and $u = y e^{i\theta}$,  this may be reformulated into
	 \begin{align*}
	 	\Re  ( v       \trh  (r, \omega) - w       \trh'  (r, \omega)  ),
	 \end{align*}
	as defined by \eqref{4eq: w+-} and \eqref{4eq: trh+-}. 
	Moreover, after the change of variables, the integral 
	 \begin{align*}
	 4	\viint_{\BSA    }     \cos  (   2  \vkappa r +  2   p \omega    ) k (\vkappa, p)     (\vkappa^2 + p^2 )  \nd  \mu (\vkappa, p) 
	 \end{align*} may be easily evaluated by  
	 \begin{align*}
	 	\int  \cos (2 r \vkappa    ) \exp \big(    - (\vkappa /K)^2 \big)  \nd \vkappa = \sqrt{\pi} K \exp \big( \hskip -1pt    -  (K { r }  ) ^2   \big) ,
	 \end{align*}
	 and (by Poisson)
	 \begin{align*}
	 	\sum  \cos (2 \omega p ) \exp \big(        -  (  p / P  )^2 \big) = \sqrt{\pi} P \sum  \exp \big( \hskip -1pt   - (P (\omega + \pi p) )^2 \big) ,
	 \end{align*}
	 yielding $f (r,  \omega)$ as given by \eqref{4eq: f (r, w)} and \eqref{4eq: defn of h, theta}.  
\end{proof}

\begin{proof}[Proof of Lemma \ref{lem: integral repn, 2}] 
In view of \eqref{4eq: integral H(z)} and \eqref{4eq: f (r, w)},  the formula in \eqref{4eq: integral H(z), 2} follows from  partial integration, along with the identity
\begin{equation*}
	\begin{split}
		    \big(   (\partial / \partial r)^2     +     (\partial / \partial \omega)^2   \big)  	& \cos  ( \Re   ( v       \trh  (r, \omega) -  w       \trh'  (r, \omega)   )   )  \\
		& =   -  g ( r, \omega; v      ,  w      ) \cdot  \cos   ( \Re     ( v       \trh  (r, \omega) -  w       \trh'  (r, \omega)   )   ) .
	\end{split}
\end{equation*}   
\end{proof}



\section{Further Analysis for the Bessel Integral}\label{sec: further analysis}

This section is for the asymptotic analysis of the variant Bessel integral   $	\SDI (v, w)$  as in \eqref{4eq: I(v;w), 1}, and may be considered as the prelude to \S \ref{sec: reductions}. However, in practice, the $\boldsymbol{a}$-average will be inside the integrals, and one has to first apply bounds for the quadratic forms of $\boldsymbol{a} $, 
so the results here are not intended for direct applications.  

Let $K = P = T$ throughout this section.

We start with the formula in  \eqref{4eq: I(v;w), 1} and simplify it in {\it four} steps:


  Step 1: Truncate smoothly the double integral in \eqref{4eq: I(v;w), 1} at $|r|, |\omega | \asymp   T^{\vepsilon} / T$, so that the   integral $	\SDI (v, w)$ is turned into 
	\begin{align*}
		  \int_{-T^{\vepsilon}/T}^{T^{\vepsilon}/T}\int_{-T^{\vepsilon}/T}^{T^{\vepsilon}/T}
		\exp (i \Re (v   ( \trh  (r, \omega) -1))) \cos  ( \Re   ( w  \trh'  (r, \omega)   )   ) f (r, \omega) \tau (r, \omega)  \nd r \nd \omega ,  
	\end{align*}
up to an error term $O \big( T^2  \exp (-T^{\vepsilon})  \big)$, where $\tau (r, \omega)$ is a suitable cut-off function of the indicated square support and bounds $$  
\frac {\partial^{\valpha+\beta} \tau (r, \omega) } {\partial r^{\valpha} \partial \omega^{\beta} }  \Lt_{\valpha, \beta}    T^{\gamma-\vepsilon \gamma} , \qquad \gamma = \valpha + \beta   .$$

  Step 2: In the above integral, approximate the weight $f (r, \omega)$ (see \eqref{4eq: f (r, w)} and \eqref{4eq: defn of h, theta}) by
\begin{align}\label{5eq: f (r,w), 2}
	f_{\natural} (r, \omega) =  
	- k ''(r) k  (\omega) - k  (r) k '' (\omega) ,
\end{align}
where    \begin{align}\label{5eq: h(r), h(w)}
	k (r) = \sqrt{\pi} T \exp \left(- (T r  )^2 \right), \qquad  k (\omega) = \sqrt{\pi} T \exp \left(- (T\omega )^2 \right); 
\end{align}  the error term is dominated by that in Step 1.

\begin{lem}
	Define 
	\begin{align}\label{5eq: I(v, w), 2}
	 \SDI_{\natural} (v, w)  \hskip -1pt  = \hskip -2pt  \viint  \hskip -1pt
	\exp (i \Re (v   ( \trh  (r, \omega) -1))) \cos  ( \Re   ( w  \trh'  (r, \omega)   )   ) f_{\natural} (r, \omega) \tau (r, \omega)  \nd r \nd \omega . 
	\end{align}
We have the bound 
\begin{align}\label{4eq: bound for H(z;u)}
	\SDI_{\natural} (v, w)  \Lt_{\gamma} T^2   \bigg( \frac { T  } {  |w|  } 
	+   \frac 1 {T}    \frac {|v|^2} {|w|^3} \bigg)^{2\gamma} 
\end{align}
for any integer $\gamma \geqslant 0$.  
\end{lem}

\begin{proof}
Recall the definition of $g ( r, \omega; v      ,  w      ) $ as in \eqref{4eq: f (w; r w)}. We introduce the differential operators 
	\begin{align*}
		\mathrm{D} = - \frac {1} { g ( r, \omega; v      ,  w      ) } \bigg( \frac {\partial^2} {\partial r^2} + \frac {\partial^2} {\partial \omega^2} \bigg), \qquad \mathrm{D}^* =  - \bigg( \frac {\partial^2} {\partial r^2} + \frac {\partial^2} {\partial \omega^2} \bigg) \frac {1} { g ( r, \omega; v      ,  w      ) };
	\end{align*}
	$\mathrm{D}^{*}$ is the dual of $\mathrm{D}$. Note that $\exp (i \Re (v   ( \trh  (r, \omega) -1))) \cos  ( \Re   ( w  \trh'  (r, \omega)   )   )$ is invariant under $\mathrm{D}$ (see the proof of Lemma \ref{lem: integral repn, 2}), so   $\mathrm{D}^*$ may be executed repeatedly.  By induction we may prove that $\mathrm{D}^{*  \gamma}  (f_{\natural} (r, \omega) \tau (r, \omega)  ) $  is a linear combination of quotients of the form
	\begin{align*}
		\frac { \partial^{\valpha}   (f_{\natural} (r, \omega) \tau (r, \omega)  )  \prod    \lp \partial^{\beta} g ( r, \omega; v      ,  w      ) \rp^{k_{\beta}}  } {g( r, \omega; v      ,  w      )^{\gamma+k} } , \qquad {\sum} k_{\beta} = k , \quad  \valpha + \sum  k_{\beta} \beta  = 2 \gamma,  
	\end{align*}
	where $  \partial^{\valpha}$ is the shorthand for any partial derivative of degree $ \valpha $. It is then an easy excise to derive  \eqref{4eq: bound for H(z;u)}. 
	Note that $g ( r, \omega; v      ,  w      )  $ is dominated by $|w|^2$ due to the $\cosh^2 r$ and that   its {\it first} derivative   $ \partial  g ( r, \omega; v      ,  w      )  $  contains  either $  \sinh  2r \cdot |v|^2  $ or $\sin  2 \omega \cdot |v|^2 $   (hence the factor $ 1/T$).   
\end{proof}

  Step 3: 
In view of (see  \eqref{4eq: trh+-}) 
\begin{align*}
	   \trh  (r, \omega) = 1  + O \big( r^2 + \omega^2 \big),  \qquad    \trh'  (r, \omega) = r + i \omega + O \big( r^3 + \omega^3 \big), 
\end{align*}
further replace   $ \exp (i \Re (v   ( \trh  (r, \omega) -1))) \cos  ( \Re   ( w  \trh'  (r, \omega)   )   ) $ by    
\begin{align}\label{5eq: approx}
	  \cos   (   \Re   (   w       (r + i \omega)    )   )   + O  \big(  (|v      | + | w      |  ) \big(r^2 + \omega^2 \big) \big), 
\end{align}
so that the integral $\SDI_{\natural} (v, w)$ is  simplified into 
\begin{align*} 	
	\viint	\cos    (\Re (w       (r + i \omega) ))        f_{\natural} (r, \omega) \tau (r, \omega)   \nd r \nd \omega ,  
\end{align*}
up an error term $O  (|v      | + | w      |  )   $.

  Step 4: Extend the integral domain onto $\BR^2$, with the cut-off function $\tau (r, \omega) $ removed and an error term added as in Step 1, and conclude by a simple evaluation of the double Fourier integral:
\begin{align*} 
	 	\viint 		\cos (\Re (w       (r + i \omega) ))     f_{\natural} (r, \omega)    \nd r \nd \omega =    |  \pi w      |^2 \exp     \lp     -  (   { |w|} / {2 T}  )^2    \rp .   
\end{align*}  

Note that only Step 3 contributes an essential error term $O  (|v      | + | w      |  )   $. 

\begin{lem}\label{lem: integral repn, asymp}
	We have the asymptotic formula 
	\begin{equation}\label{5eq: integral H(z), asym}
		\begin{split}
			\SDI (v, w) =    |  \pi w      |^2 h (w/2)  + O  \big(     |v      | + | w      |   +  T^2    \exp (-T^{\vepsilon})   \big), 
		\end{split} 
	\end{equation}
where 
\begin{align} \label{5eq: h(w)}
	h (z) = \exp \left(  -   \frac {|z|^2 } {T^2}   \right) . 
\end{align} 
\end{lem}


\section{Quadratic Forms and Large Sieve} \label{sec: quad} 

In this section,  we recollect simple bounds for quadratic forms with Kloosterman sums and establish   hybrid variants of the large sieve inequality.  Subsequently, $\varPi, \varLambda, M, N, C $ will always denote  positive parameters.

\subsection{Bounds for Quadratic Forms} 



\begin{lem}\label{lem: quad form}
	
	Let $\boldsymbol{b} : \CaloO     \ra \BC$ and $ c \in \CaloO \smallsetminus \{0\}$.  
	We have 
	\begin{equation} 
		\label{6eq: Kloosterman, 1.1}     
		\mathop{{\mathop{\sum \sum}_{ {M} < |m|  \leqslant    {M} +  {\varPi} } }}_{ \,  {N} <   |n|  \leqslant    { N} +  {\varLambda} }   \left| b_{m} \overline{b }_{n}   S (m, n; c) \right|    \Lt   \tau^2 (c) |c |  \cdot \varPi  \|\boldsymbol{b}_{M, \varPi} \|_2^{} \cdot   \varLambda    \|\boldsymbol{b}_{N, \varLambda} \|_2^{}  ,  
	\end{equation} 
	\begin{align} 
	\label{6eq: Kloosterman, 2.1}       \mathop{{\mathop{\sum \sum}_{ {M} < |m|  \leqslant    {M} +  {\varPi} } }}_{ \,  {N} <   |n|  \leqslant    { N} +  {\varLambda} } b_{m} \overline{b }_{n}   S (m, n; c)   \Lt   \sqrt{|c|^2 + \varPi^2}  \|\boldsymbol{b}_{M, \varPi} \|_2 \cdot \sqrt{|c|^2 + \varLambda^2}     \|\boldsymbol{b}_{N, \varLambda} \|_2  ,  
\end{align}   
	and in particular 
	\begin{align}
		\label{6eq: Kloosterman, 1}  	   {\mathop{\sum \sum}_{ {N} < |m| , |n|  \leqslant    {2N} } } \left| b_{m} \overline{b }_{n}   S (m, n; c) \right|   \Lt \tau^2 (c) |c| N^2  \|\boldsymbol{b}_N \|_2^2, 
	\end{align} 
\begin{align} 
	\label{6eq: Kloosterman, 2}   \  {\mathop{\sum \sum}_{ {N} < |m| , |n|  \leqslant    {2N} } } b_{m} \overline{b }_{n}   S (m, n; c)   \Lt  \big(|c|^2 + N^2    \big) \|\boldsymbol{b}_N \|_2^2 ,
\end{align}  
where
\begin{align*}
	\|\boldsymbol{b}_{N, \varLambda} \|_2 = \bigg(  \sum_{  {N} <  |n|  \leqslant    {N} +  {\varLambda} } |b_{n}|^2 \bigg)^{1/2}, \qquad  
		\|\boldsymbol{b}_{N} \|_2^{} = \bigg(  \sum_{  {N} <  |n|  \leqslant    {2N}  } |b_{n}|^2 \bigg)^{1/2} .
\end{align*}
\end{lem}

\begin{proof}
	The   inequalities \eqref{6eq: Kloosterman, 1.1} and \eqref{6eq: Kloosterman, 1} follow  from Cauchy--Schwarz and the Weil  bound:
	\begin{align*}
		S (m, n; c)   \Lt \tau (c) {|(m, n, c)|} {|c|}. 
	\end{align*} 
	The inequalities \eqref{6eq: Kloosterman, 2.1} and \eqref{6eq: Kloosterman, 2}   follow  from Cauchy--Schwarz and the mean value theorem:  
	\begin{align*} 
		\sum_{\valpha (\mod c)}     \bigg|   \sum_{  {N} < |n| \leqslant     {2N}  }  b_{n}   e \Big[  \frac {\valpha n} {c} \Big] \bigg|^2  \Lt  \big(|c|^2   +  {\varLambda^2 } \big) \sum_{  {N} < |n| \leqslant     {2N}  }  |b_{n}|^2 ,
	\end{align*} 
	which is an easy consequence of the orthogonality relation.
\end{proof}

Let us also record here a useful inequality for the Ramanujan sums: 
	\begin{align}\label{5eq: Ramanujan}
		\sum_{|c| \leqslant C} \bigg( \sum_{  {N} < |n| \leqslant     {2N}  }  |b_{n} S(n, 0; c)|  \bigg)^2  \Lt C^2 N^{2+\vepsilon} \sum_{  {N} < |n| \leqslant     {2N}  }  |b_n |^2, 
	\end{align} 
which is a direct consequence of $ |S(n, 0;c)| \leqslant |(n,c)|^2 $ and Cauchy--Schwarz. 

\subsection{A Hybrid Large Sieve} 

The large sieve inequality over $\RF = \BQ(i)$ reads (see \cite[Theorem 2]{Huxley-LS-Num-Fields}): 
\begin{align}\label{11eq: large sieve}
	\sum_{|c| \leqslant C}  \ 	\sumx_{\valpha (\mod c)}  \bigg|   \sum_{  {N} < |n| \leqslant     {N} + {\varLambda}}  b_{n}   e \Big[  \frac {\valpha n} {c} \Big] \bigg|^2  \Lt  \big(C^4   +  {\varLambda^2 } \big) \sum_{  {N} < |n| \leqslant     {N} + {\varLambda}}  |b_{n}|^2,
\end{align}  
and we wish to prove a hybrid variant of \eqref{11eq: large sieve} as follows. It is not hard to generalize this into a form analogous to \cite[Lemma 6]{Young-GL(3)-Special-Points}.

\begin{lem}\label{lem: hybrid LS}
	Let $\boldsymbol{b} : \CaloO     \ra \BC$  and $ v \in \BC \smallsetminus \{0\}$.   
	We have 
	\begin{align}\label{11eq: hybrid LS}
		\viint_{\overbar{D}_{\rho}}  \sum_{|c| \leqslant C}  \ 	\sumx_{\valpha (\mod c)}  \bigg|   \sum_{  {N} < |n| \leqslant     {N} + {\varLambda}}  b_{n}   e \Big[  \frac {\valpha n} {c} \Big] e \Big[\frac {n z} {v} \Big]\bigg|^2 \nd z \Lt \big(C^4 \rho^2 +   |v|^2 \big)  \|\boldsymbol{b}_{N, \varLambda} \|_2^2,
	\end{align} 
	where $D_{\rho}  $  is the disc  of radius $\rho$ centered at the origin. 
\end{lem}

\begin{proof}
	For the proof we adopt some  arguments of Luo \cite[\S 3]{Luo-Twisted-LS} and Young \cite[\S 6]{Young-GL(3)-Special-Points}.  
	
	By the change of variables $z \ra z  \rho$ and $v \ra v  \rho$, it suffices to consider the case $\rho = 1$. Therefore, if we write $D = D_1$, then our goal is to prove
	\begin{align}\label{11eq: T=1}
		\viint_{\overbar{D} }  \sum_{|c| \leqslant C}  \ 	\sumx_{\valpha (\mod c)}  \bigg|   \sum_{  {N} < |n| \leqslant     {N} + {\varLambda}}  b_{n}   e \Big[  \frac {\valpha n} {c} \Big] e \Big[\frac {n z} {v} \Big]\bigg|^2 \nd z \Lt \big(C^4  +   |v|^2 \big)  \|\boldsymbol{b}_{N, \varLambda} \|_2^2. 
	\end{align}   
	Note that  \eqref{11eq: T=1} follows directly from \eqref{11eq: large sieve} if $\varLambda \Lt |v|$, so we only need to consider the case when $|v| \Lt  \varLambda$. 
	
	Let $\phi$ be a non-negative Schwartz function on $\BC$ such that $\phi  \geqslant 1$ on the unit disc $\overbar{D}$ and that $\widehat{\phi}$ is of compact support. Such $\phi$ may be constructed by product from the Schwartz functions on $\BR$ as in \cite{Vaaler-Fourier}. Thus for  $\boldsymbol{c} : \CaloO     \ra \BC$ supported on $ N < |n| \leqslant N+\varLambda $ we have
	\begin{align*}
		\viint_{D }    \bigg|   \sum_{  n }  c_{n} e \Big[\frac {n z} {v} \Big]\bigg|^2 \nd z \leqslant  	\viint   \phi (z)  \bigg|   \sum_{  n }  c_{n} e \Big[\frac {n z} {v} \Big]\bigg|^2  \nd z 
		= \mathop{\sum \sum}_{m, n} {c}_{m} \overline{c}_{n} \widehat{\phi} \Big( \hskip -1pt        - \frac {m-n} {v} \Big). 
	\end{align*}
	Since $\widehat{\phi}$ is compactly supported, we must have $|m-n| \Lt |v|$. As in \S \ref{sec: reductions}, we dissect the sum over $m$ and $n$ into hyper-annuli $$A_{IJ} = \big\{ (m, n) : |m| \in I, |n| \in J \big\}, $$
	for $I$ and $J$ sub-intervals of $( {N},  {N+\varLambda}]$ of equal length $\varDelta = O (|v|)$ ($\varDelta \leqslant \varLambda$). The only relevant hyper-annuli  $ A_{IJ}$ are those with $I$ and $J$ either equal or adjacent. Next we reverse the Fourier transform so that the partial sum on such $A_{IJ}$ is turned into
	\begin{align*}
		\viint	\phi (z)   \mathop{\sum \sum}_{(m, n)\in A_{I J} } {c}_{m} \overline{c}_{n} e \Big[   \frac {m-n} {v} \Big]  \nd z.  
	\end{align*} 
	Further, as there are at most three such $A_{IJ}$ for each $I$,  by Cauchy we infer that 
	\begin{align*}
		\viint_{D }    \bigg|   \sum_{ N < |n| \leqslant N+\varLambda }  c_{n} e \Big[\frac {n z} {v} \Big]\bigg|^2 \nd z \leqslant 3  \viint   \phi (z) \sum_{I}   \bigg|   \sum_{  {N_I} < |n| \leqslant     {N_I} + {\varDelta}  }  c_{n} e \Big[\frac {n z} {v} \Big]\bigg|^2 \nd z, 
	\end{align*}
	where $I = (N_I, N_I +\varDelta ]$, say. Finally, specialize this to $ c_{n} = b_{n} e  [    {\valpha n} /{c}  ]  $ and sum over $\valpha $ and $c$, then \eqref{11eq: T=1} follows from \eqref{11eq: large sieve}, applied with $N$ and $\varLambda$ replaced by $N_I$ and $ \varDelta = O (|v|)$. 
\end{proof}

\begin{cor}\label{cor: hybrid LS}
	We have
	\begin{align}\label{11eq: hybrid LS, 1}
	&	\viint_{\overbar{D}_{\rho} }  \sum_{   |c| \leqslant   C }  \,	\sumx_{\valpha (\mod c)}  \bigg|   \sum_{  {N} < |n| \leqslant     {2N}  }  b_{n}   e \Big[  \frac {\valpha n} {c} \Big] e \Big[\frac {n z} {c v} \Big]\bigg|^2  \nd z \Lt \rho^2 \big(C^4   + C^{\vepsilon} N^2 \big)  \|\boldsymbol{b}_{N} \|_2^2,  \\
		& \label{11eq: hybrid LS, 2}
		\viint_{\overbar{D}_{\rho} }  \sum_{   |c| \leqslant   C }  \,	\sumx_{\valpha (\mod c)}  \bigg|   \sum_{  {N} < |n| \leqslant     {2N}  }  b_{n}   e \Big[  \frac {\valpha n} {c} \Big] e \Big[\frac {n z} {c v} \Big]\bigg|^2  \nd z \Lt \big(C^4 \rho^2  +  C^2 |  v|^2 \big)  \|\boldsymbol{b}_{N} \|_2^2. 
	\end{align}    
\end{cor}

\begin{proof}
	By the change $z \ra c z$ we rewrite the expression on the left  as
	\begin{align*}
		\sum_{   |c| \leqslant   C }  |c|^2	\viint_{\overbar{D}_{\rho/|c|}} \,	\sumx_{\valpha (\mod c)}  \bigg|   \sum_{  {N} < |n| \leqslant     {2N}  }  b_{n}   e \Big[  \frac {\valpha n} {c} \Big] e \Big[\frac {n z} {v} \Big]\bigg|^2  \nd z , 
	\end{align*}  then via a dyadic partition \eqref{11eq: hybrid LS, 1} and \eqref{11eq: hybrid LS, 2} follow  from  \eqref{11eq: large sieve} and \eqref{11eq: hybrid LS} respectively. 
\end{proof}

\delete{
\begin{lem} \label{lem: quad form, 2}
		Let $\boldsymbol{b} : \CaloO   \ra \BC$  and $ q \in \CaloO \smallsetminus \{0\}$. Let $f (z) \in C (D_{\rho})$ be   bounded.  
	Then  
	\begin{align} 
		\label{6eq: quad V, 3.1}     
	&	  \sum_{   |c| \leqslant  C}      {\mathop{\sum \sum}_{ {N} < |m| , |n|  \leqslant    {2N} } }    b_{m} \overline{b }_{n}   V_q (m, n; c)   \viint_{D_{\rho}}   e  \Big[     \frac {  m - n }   {cq} z \Big]   f(z) \nd z 	\Lt \rho^2   \big({C^4 +  C^{\vepsilon} N^2}\big) \|\boldsymbol{b}_{N} \|_2^2   ,  \\
	\label{6eq: quad V, 3.2}     
	&	\sum_{    |c| \leqslant C}      {\mathop{\sum \sum}_{ {N} < |m| , |n|  \leqslant    {2N} } }    b_{m} \overline{b }_{n}   V_q (m, n; c)   \viint_{D_{\rho}}   e  \Big[     \frac {  m - n }   {cq} z \Big]   f(z) \nd z 	\Lt    \big({C^4 \rho^2 +  C^2 |q|^2}\big) \|\boldsymbol{b}_{N} \|_2^2   .   
\end{align} 
\end{lem}

\begin{proof}
	 After opening $V_q (m, n; c)$ by \eqref{1eq: defn V}, this lemma is a direct consequence of Cauchy--Schwarz and Corollary \ref{cor: hybrid LS}.  
\end{proof}
}

\section{Application of Kuznetsov} \label{sec: apply Kuznetsov}

We start by extending the spectral sum onto $\Pi_{c}$ (see the remark  below \eqref{2eq: sum, cusp}	 and \eqref{2eq: sum, Eis}),  choosing the spectral weight in   \eqref{1eq: Kuznetsov}  to be $h (\vkappa, p; \sqrt{m/n})$ as defined in \eqref{4eq: choice of h}, multiplying both sides of \eqref{1eq: Kuznetsov}  by ${a}_{(m)}  {a}_{(n)} $, and then summing over $m, n \in \CaloO$ with $  {N} < |m|, |n| \leqslant  {2N} $.  Note that $$\Re ( \vchi_{i \vkappa,  p} (m/n) ) =   \cos  (  \vkappa \log |m/n| +   p \arg (m/n)  ), $$
so the Kuznetsov formula \eqref{1eq: Kuznetsov}   in Lemma \ref{lem: Kuznetsov}  yields
\begin{align}\label{6eq: after Kuz}
	\SC  (\boldsymbol{a}) + \SE    (\boldsymbol{a}) =  \frac {1} {32 \pi^3 } P (\boldsymbol{a}) + O \big( KP \big(K^2+P^2\big) 	\| \boldsymbol{a}_N \|_2^2 \big)  , 
\end{align}
with
\begin{align}\label{6eq: P(a)}
	P (\boldsymbol{a}) =  \sum_{c}  	\mathop{\sum\sum}_{m, n}   {a}_{(m)}  {a}_{(n)} \frac{S (m, n;c)} {|c|^2} \SDH \bigg( \frac {2\pi \sqrt{mn}} {c} ; \sqrt{\frac {m} {n} } \bigg) .  
\end{align} 
It follows from $m \leftrightarrow n$ that 
\begin{align}\label{6eq: P(a), 2}
	P (\boldsymbol{a}) =   \Re      \sum_{c}  \frac 1 {|c|^2}	\mathop{\sum  \sum}_{m, n}   {a}_{(m)}  {a}_{(n)}  {S (m, n;c) e \Big[\frac {m+n} {c}     \Big]}  \SDI \bigg(     2\pi \frac {m      +     n} {c},   2\pi \frac {m  -      n} {c}    \bigg) .  
\end{align}
See \eqref{5eq: Bessel H(z)}--\eqref{4eq: v, w = m+-n} for the expressions of the Bessel $\SDH$ and $\SDI$ integrals.

\section{Proof of Theorem \ref{thm: prelim}} \label{sec: proof prelim}
As remarked below \eqref{4eq: v, w = m+-n},   $m$ and $n$ become separate in the (variant) Bessel integral, so Lemma \ref{lem: quad form} may be applied easily inside the integrals. 

\delete{Define $	\varPhi_q (c; \boldsymbol{a}) $ to be the inner $(m, n)$-sum in \eqref{6eq: P(a), 3}:
	\begin{align}\label{8eq: defn of Phi}
		\varPhi_q (c; \boldsymbol{a}) = \mathop{\sum\sum}_{m, n}   {a}_{(m)}  {a}_{(n)}  {V_q (m, n;c)}   \SDI \bigg(      2\pi \frac {m  +   n} {c q},   2\pi \frac {m   -    n} {c q} \hskip -1pt \bigg) . 
\end{align} }

Define  $\varPhi (c; \boldsymbol{a})$  to be the inner $(m, n)$-sum in \eqref{6eq: P(a), 2}:
	\begin{align}\label{8eq: Phi}
		\varPhi (c; \boldsymbol{a}) =	\mathop{\sum   \sum}_{m, n}   {a}_{(m)}  {a}_{(n)}  {S (m, n;c) e \Big[\frac {m+n} {c}     \Big]}  \SDI \bigg(      2\pi \frac {m      +     n} {c},   2\pi \frac {m  -      n} {c}   \bigg) . 
	\end{align}

If $|c | > N^2$ then by   \eqref{4eq: I(v;w), 2}  and \eqref{6eq: Kloosterman, 1} we get
\begin{align}\label{7eq: bound for Phi(c;a), 1}
	\varPhi (c; \boldsymbol{a})   \Lt   \frac {\tau^2 (c) N^4} {|c| }   \|\boldsymbol{a}_N \|_2^2. 
\end{align}

If $ {N} < |c| \leqslant N^2$ then by   \eqref{4eq: I(v;w), 2} and \eqref{6eq: Kloosterman, 2} we get
\begin{align}\label{7eq: bound for Phi(c;a), 2}
	\varPhi (c; \boldsymbol{a})   \Lt   N^2   \|\boldsymbol{a}_N \|_2^2. 
\end{align}

If $|c| \leqslant  {N}    $ then by   \eqref{4eq: I(v;w), 1} and \eqref{6eq: Kloosterman, 2}  we get
\begin{align}\label{7eq: bound for Phi(c;a), 3}
	\varPhi (c; \boldsymbol{a})  \Lt  \big(K^2+P^2\big)  N^2 \|\boldsymbol{a}_N \|_2^2 . 
\end{align}


By the estimates \eqref{7eq: bound for Phi(c;a), 1} and \eqref{7eq: bound for Phi(c;a), 2},  it follows from \eqref{6eq: P(a), 2} that
\begin{align}\label{8eq: c < N}
	  P    (\boldsymbol{a}) =   \Re      \sum_{ |c| \leqslant N}   \frac {\varPhi  (c; \boldsymbol{a})} {|c|^2}   + O \big(   N^{2+\vepsilon}  	\| \boldsymbol{a}_N \|_2^2 \big) . 
\end{align}

Moreover, if \eqref{7eq: bound for Phi(c;a), 3} is taken into account, then by   \eqref{6eq: after Kuz} we have
	\begin{align}
		 \SC  (\boldsymbol{a}) + \SE (\boldsymbol{a}) \Lt \big(K^2 + P^2\big) \big(  K P +   { N^{2+\vepsilon}}    \big) \|\boldsymbol{a}_{N} \|_2^2, 
	\end{align}
and consequently Theorem \ref{thm: prelim}.

\section{Reductions for the Application of Poisson}  \label{sec: reductions}

In the rest of this paper, we shall always assume $K = P = T$. 

In this section, we proceed as in \S \ref{sec: further analysis} to extract a main term from 	$P (\boldsymbol{a})$, with the aid of   Lemmas \ref{lem: quad form}   for the error estimation. 

As the main term of $ \SDI (v, w)  $ is equal to $  |  \pi w      |^2 h (w/2)$ (see \eqref{5eq: integral H(z), asym} and \eqref{5eq: h(w)}), so that of $P (\boldsymbol{a})$ reads: 
\begin{align}\label{8eq: Q(a), 0} 
	Q (\boldsymbol{a}) =	  {4 \pi^4}   \Re \sum_{c}  \frac 1 {|c |^4}	\mathop{\sum\sum}_{m, n}   {a}_{(m)}  {a}_{(n)} |m-n|^2    {S (m, n;c)e \Big[\frac {m+n} {c}     \Big]}    h \Big(  \pi \frac {m-n} {c} \Big). 
\end{align}
Similar to \eqref{8eq: c < N}, one may restrict to $|c| \leqslant N$ at the cost of an error $O \big(   N^{2+\vepsilon}  	\| \boldsymbol{a}_N \|_2^2 \big)$, by using
\begin{align}\label{8eq: h (w) = int, 1}
	h (w/2) = \viint 		\cos (\Re (w       (r + i \omega) ))    k  (r) k (\omega)  \nd r \nd \omega
\end{align}
to separate $m$ and $n$ (see \eqref{5eq: h(r), h(w)} and \eqref{5eq: h(w)}).

By \eqref{6eq: Kloosterman, 1}  the error terms in Steps 1, 2, and 4 are all $O \big(    N^{3+\vepsilon}/T^2 \cdot \|\boldsymbol{a}_{N} \|_2^2 \big)$, say.  
However,  Step 3 requires   more delicate analysis. To this end,  some arguments from  \cite[\S 3]{Luo-Twisted-LS} will be adopted. 

Let $E  (c; \boldsymbol{a})$ denote the error-term contribution to $\varPhi  (c; \boldsymbol{a})$ as in Step 3.

For $ N^{1-\vepsilon} /T  < |c| \leqslant N$,  in view of     \eqref{5eq: f (r,w), 2}--\eqref{5eq: I(v, w), 2},  \eqref{5eq: approx}, and \eqref{8eq: Phi},    by \eqref{6eq: Kloosterman, 2} we infer that 
\begin{align*}
E  (c; \boldsymbol{a}) \Lt 	 	 \frac {N} {|c|} \cdot  N^2 \|\boldsymbol{a}_{N} \|_2^2  <  T N^{2+\vepsilon}  \|\boldsymbol{a}_{N} \|_2^2 . 
\end{align*}
Here one needs to use   Taylor expansion to keep $m$ and $n$ separate, for example, as follows:
\begin{align*}
	e \Big[     \frac {m+n} {c} (\trh (r, \omega)-1) \Big] =   1 +  \sum_{k=1}^{\infty}  \frac {  (2\pi i)^k } {k!} \bigg(  \Re \Big( \frac {m+n} {c} (\trh (r, \omega)-1) \Big)   \bigg)^k   . 
\end{align*}
This argument  only works effectively  when 
\begin{align*}
	\frac {N} {|c| T^2 }      < 1,
\end{align*}
so that the Taylor series are rapidly convergent.


For $ |c| \leqslant N^{1-\vepsilon} / T   $ define 
\begin{align}\label{7eq: defn of H}
	\varDelta =      \max \left\{ {|c|} T     N^{ \vepsilon}  , \, \sqrt[3]{ \frac  {  |c| N^{2+\vepsilon} } {T} }   \right\}   . 
\end{align} 
Next, we split the $m$- and $n$-sums in $ E (c; \boldsymbol{a}) $    into hyper-annuli $$A_{IJ} = \big\{ (m, n) : |m| \in I, |n| \in J \big\}, $$
where $I$ and $J$ range over sub-intervals of $( {N},  {2N}]$ of equal length $ {\varDelta} $.  If $I$ and $J$ are neither equal nor adjacent then we have $| m-n | > \varDelta $, and it follows from \eqref{4eq: bound for H(z;u)} and the choice of $\varDelta$ in \eqref{7eq: defn of H} that   $ \SDI_{\natural} (2\pi (m+n)/c   , 2\pi (m-n)/ c  )  $ is negligibly small. Then by \eqref{6eq: Kloosterman, 1} the contribution to $E (c; \boldsymbol{a})$ of such hyper-annuli  may be trivially bounded by $ O \big( \|\boldsymbol{a}_{N} \|_2^2 \big) $, say, so we only need to consider that of the remaining hyper-annuli. 
 It follows that  
\begin{align*}
E (c; \boldsymbol{a}) \Lt   	\frac {N} {|c|}     \big(|c|^2 +  \varDelta^2\big) \|\boldsymbol{a}_{N} \|_2^2     , 
\end{align*}
by \eqref{6eq: Kloosterman, 2.1} in the case $N /T^2  < |c| \leqslant   N^{1-\vepsilon} / T  $, and
\begin{align*}
E (c; \boldsymbol{a}) \Lt	 	  \tau^2 (c)  \varDelta^2 N \|\boldsymbol{a}_{N} \|_2^2 , 
\end{align*}
by  \eqref{6eq: Kloosterman, 1.1} in the case $|c|  \leqslant N  /T^2  $. Note that the sum of $ \|\boldsymbol{a}_{N_I, \varDelta } \|_2^{} \cdot        \|\boldsymbol{a}_{N_J, \varDelta } \|_2^{} $ (say $I = (N_I, N_I +\varDelta  ]$) with $I$ and $J$ equal or adjacent is bounded by $ 3 \|\boldsymbol{a}_{N} \|_2^2 $.  Moreover,  in the former case $\varDelta =  |c| T    N^{ \vepsilon}$ and hence
\begin{align*}
	E (c; \boldsymbol{a}) \Lt   {|c| T^2N^{1+\vepsilon} }    \|\boldsymbol{a}_{N} \|_2^2  \Lt   T {  N^{2+\vepsilon} }    \|\boldsymbol{a}_{N} \|_2^2 , 
\end{align*}
while in the latter case $\varDelta = \sqrt[3]{|c|N^2 / T} N^{\vepsilon}   $ and hence
\begin{align*}
	E (c; \boldsymbol{a}) \Lt  \tau^2 (c) |c|^{2/3} \frac {N^{7/3}} {T^{2/3}} N^{\vepsilon} \|\boldsymbol{a}_{N} \|_2^2 \Lt  \frac {  N^{3+\vepsilon} }   {T^2} \|\boldsymbol{a}_{N} \|_2^2 .
\end{align*}

By gathering the error bounds for $E(c; \boldsymbol{a})$ above and summing over $c$ as in \eqref{8eq: c < N},  we conclude that 
\begin{equation} \label{7eq: P=Q}
	P  (\boldsymbol{a})  =    Q (\boldsymbol{a}) + O \bigg(    \bigg( T +  \frac N {T^2} \bigg) N^{2+\vepsilon} \|\boldsymbol{a}_{N} \|_2^2 \bigg) ,
\end{equation} 
where $Q  (\boldsymbol{a}) $ is defined by \eqref{8eq: Q(a), 0}.  

{\begin{remark}\label{rem: Luo}
	Our asymptotic formula {\rm\eqref{7eq: P=Q}} is reminiscent of Luo's {\rm(36)} in \cite{Luo-Twisted-LS}. However, our error term is relatively  weaker, since the approximation $$ \frac {m+n} c \trh  (r, \omega)    = \frac {m+n} c  + O  \bigg(  \frac {N (r^2 + \omega^2)} {|c|} \bigg) $$ is {\it not} as strong as Luo's   $$    \frac {m+n} c \cos \delta   = \frac {m+n} c  + O \bigg(\frac { N \delta^2} c\bigg) . $$   
\end{remark}}

Next, by \eqref{1eq: defn Kloosterman} we write 
\begin{align*}
	S (m, n; c)  e\Big[\frac {m+n} {c} \Big] = \sumx_{   \valpha  (\mathrm{mod} \, c) } e \bigg[  \frac {  (1-\valpha) m +  (1- \widebar{\valpha}) n } {c} \bigg],
\end{align*}
and split the sum according to the ideal $(1-\valpha, c) = (q)$. Thus $c = d q$ and $\valpha = 1 - \widebar{\delta} q$, where $\delta$ ranges over residue classes in $\CaloO / d \hskip 1pt \CaloO$  such that $  (\delta (q-\delta), d) = (1)$. We obtain
\begin{align*}
	S (m, n; c)  e\Big[\frac {m+n} {c} \Big] = \frac 1 4 \sum_{dq=c} \mathop{\sum_{\delta (\mathrm{mod}\, d)}}_{ (\delta (q-\delta), d) = (1) }  e \bigg[  \frac {  \widebar{\delta} m +  \overline{q - \delta} n } {d} \bigg] . 
\end{align*}
Therefore, in view of \eqref{1eq: defn V} and  \eqref{8eq: Q(a), 0}, we have 
\begin{align}\label{8eq: Q(a)} 
	Q (\boldsymbol{a}) =	  {\pi^4}   \Re \mathop{\sum\sum}_{c, q}  \frac 1 {|c q|^4}	\mathop{\sum\sum}_{m, n}   {a}_{(m)}  {a}_{(n)} |m-n|^2     {V_{q} (m, n; c)}      h \Big(  \pi \frac {m-n} {c q} \Big). 
\end{align}

Finally, let  $ 1\leqslant X \leqslant T $ be a parameter to be chosen optimally later. Let $Q  (\boldsymbol{a}; X)$ be the partial sum of  \eqref{8eq: Q(a)}  restricted by $|c| \leqslant X$. It follows from the hybrid large sieve in Corollary \ref{cor: hybrid LS} that
\begin{align}\label{7eq: Q(a)=Q(a;X)}
	Q (\boldsymbol{a}) = Q (\boldsymbol{a}; X) + O \bigg(   \frac {T^2 N^{2 +\vepsilon}} {X^2}    \|\boldsymbol{a}_{N} \|_2^2 \bigg). 
\end{align}
 More explicitly, for $ X \leqslant C \leqslant N /2 $, we denote by   $Q_C (\boldsymbol{a})$ the partial sum over $C < |c| \leqslant 2 C$.  
Then we separate $m$ and $n$ by \eqref{8eq: h (w) = int, 1}, open $V_q (m,n; c)$ by \eqref{1eq: defn V},  and apply Corollary \ref{cor: hybrid LS} with $\rho =  N^{\vepsilon}/T$ to deduce  
\begin{align*}
	Q_C    (\boldsymbol{a})     & \Lt \sum_{q} \frac {N^2} {|C q|^4}  T^2  \bigg( \frac  {C^4}{T^2} + \min \bigg\{ \frac {N^2} {T^2},  C^2 |q|^2  \bigg\}      \bigg) N^{\vepsilon} \|\boldsymbol{a}_{N} \|_2^2 \Lt \bigg( 1  + \frac {T^2} {C^2} \bigg)  N^{2+\vepsilon} \|\boldsymbol{a}_{N} \|_2^2 . 
\end{align*}
 
 
 \section{Application of Poisson} \label{sec: Poisson}
 
Let us rearrange the sum $Q (\boldsymbol{a}; X)$ as follows
 	\begin{align*}
 		Q  (\boldsymbol{a}; X) =	\pi^4  \Re  \sum_{|c| \leqslant X}  \frac 1 {|c |^4}	\mathop{\sum\sum}_{m, n}   {a}_{(m)}  {a}_{(n)} \cdot |m-n|^2  \sum_{q}     \frac  {V_{q} (m, n; c)}  {| q |^4}   h \Big(  \pi \frac {m-n} {c q} \Big). 
 	\end{align*} 
 Following   Iwaniec,   Li, and   Young \cite{Iwaniec-Li-Ortho,Young-GL(3)-Special-Points}, we now proceed to apply the Poisson summation for the $q$-sum modulo $c$. For important technical reasons we first introduce a real-valued cut-off function $\eta \in C^{\infty} (0, \infty)$ with $\eta \equiv 0$ on $(0, 1/2]$ and $ \eta \equiv 1 $ on $[1, \infty)$ so that   the inner $q$-sum  may be rewritten as
\begin{align*}
 |m-n|^2 \sum_{  q }     {V_{q} (m, n; c)}   \frac {\eta (|q|)} {|  q|^4}   h \Big(  \pi \frac {m-n} {c q} \Big). 
\end{align*} Then we use Poisson to transform it into
 \begin{align*}
 	 \Big|   \frac {m-n} {c  }   \Big|^2 \sum_{\valpha (\mathrm{mod}\, c)} V_{\valpha} (m,n ; c)   \sum_{ q }  e \Big[ \frac {\valpha q} {c} \Big] \viint  \eta (|z|) h   \Big(  \pi \frac {m-n} {c z} \Big) e  \Big[-   \frac {q z} {c}    \Big]   \frac {\nd z} {|z|^4} .
 \end{align*}
Note that 
\begin{align*}
	\sum_{\valpha (\mathrm{mod}\, c)} V_{\valpha} (m,n ; c) e \Big[ \frac {\valpha q} {c} \Big] = S (m, q; c) S(n, q; c), 
\end{align*}
since by \eqref{1eq: defn V} it equals the double exponential sum 
\begin{align*}
	\mathop{\mathop{\sum \sum}_{\valpha, \beta (\mathrm{mod}\, c)}}_{ (\beta (\valpha-\beta), c) = (1) }   e \bigg[  \frac {  \widebar{\beta} m +  \overline{\valpha - \beta} n + \valpha q} {c} \bigg]  = {\mathop{\sumx \sumx}_{ \beta, \gamma (\mathrm{mod}\, c)}}  e \bigg[  \frac {  \widebar{\beta} m +  \overline{\gamma} n + (\beta + \gamma) q} {c} \bigg].
\end{align*} 
Let us introduce the Fourier integral
\begin{align}\label{9eq: f (w;v)}
	f (w; v ) =  {|w |^2}   \viint  \eta (|z|) h   (   w/z    ) e  [-   { v z}    ]   \frac {\nd z} {|z|^4}  . 
\end{align} 
Note that $2 \Re f (w; v) = f (w; v) + f (-w; - v)$.  

Consequently,  we obtain after Poisson:
\begin{align}\label{8eq: Q(a)=Z(a)+S(a)}
	Q  (\boldsymbol{a}; X) = Z (\boldsymbol{a}; X) + S (\boldsymbol{a}; X), 
\end{align}
where $Z (\boldsymbol{a}; X)$ is the zero frequency
\begin{align}
	Z (\boldsymbol{a}; X) = 	\pi^2  \sum_{|c| \leqslant X}  \frac 1 {|c |^4}	\mathop{\sum\sum}_{m, n}   {a}_{(m)}  {a}_{(n)} S (m, 0; c) S(n, 0; c) f \Big(   \pi \frac {m-n} { c}; 0 \Big), 
\end{align}
and $S (\boldsymbol{a}; X)$ is the dual sum 
\begin{align}\label{9eq: S (a;X)}
	S (\boldsymbol{a}; X) = \pi^2    \sum_{|c| \leqslant X}  \frac 1 {|c |^4}	\mathop{\sum\sum}_{m, n}   {a}_{(m)}  {a}_{(n)} \sum_{q \neq 0} S (m, q; c) S(n, q; c) f \Big(   \pi \frac {m-n} { c}; \frac {q} {c} \Big) . 
\end{align}

Next, our aim is to establish the following asymptotic results on  $Z (\boldsymbol{a}; X) $ and $S (\boldsymbol{a}; X) $. This will be done  in \S \S  \ref{sec: zero frequency} and \ref{sec: dual sum}. 
 
\begin{prop}\label{prop: Z}
Let $ 1 \leqslant X \leqslant T < N $. 	We have
	\begin{align}\label{9eq: Z, asymp}
		Z (\boldsymbol{a}; X) = \pi^3   \varSigma (\boldsymbol{a}) + O \bigg(  \bigg(\frac {T^2 N^{2 } } {X^2} + T^4\bigg) N^{\vepsilon} \|\boldsymbol{a}_{N} \|_2^2 \bigg), 
	\end{align}
where $\varSigma (\boldsymbol{a})$ is defined by {\rm\eqref{2eq: Sigma (a)}}. 
\end{prop}

\begin{prop}\label{prop: S}
Let $ 1 \leqslant X \leqslant T < N $. 	 We have 
	 \begin{align}\label{8eq: S(a;X)=S0(a;X)}
	 	S (\boldsymbol{a}; X) = S_0 (\boldsymbol{a}; X) + O   \big(  \|\boldsymbol{a}_{N} \|_2^2 \big), 
	 \end{align}
 where
 \begin{equation}\label{8eq: S0(a;X)}
 	\begin{split}
 		S_0 (\boldsymbol{a}; X) \Lt   T^4 N^{\vepsilon}   \sum_{|c| \leqslant X}  \frac 1 {|c|^2}  	\sum_{0 < |q| \leqslant |c|   N^{\vepsilon}     }   \frac 1 {|q|^2}    \viint_{\overbar{D}_{  N^{\vepsilon}/T }}    \bigg|  \sum_{ n}    {a}_{(n)}  S (n, q; c) e \Big[ \frac {n} {c} u  \Big] \bigg|^2 \hskip -1pt \nd u ,
 	\end{split}
 \end{equation}
and hence
\begin{align}	\label{8eq: S(a;X)}
	S (\boldsymbol{a}; X) \Lt  T^4   X^2   N^{\vepsilon} \|\boldsymbol{a}_{N} \|_2^2 .
\end{align}
\end{prop}

\section{Analysis for the Fourier Integrals} 

For the separation of $m$ and $n$ in $w  = \pi (m-n)/ c$, we need to use the Fourier method. To this end, we  prove here some estimates for $ f (w; v)$ and $\widehat{f} (u; v)$.

\begin{lem}\label{lem: bounds for f}   We have
	\begin{equation}\label{10eq: bound for f, 1}
		 f (w; v) \Lt_{\gamma} \min \bigg\{ \frac {|w|^2} {|v|^{ 2\gamma}}, T^2 \bigg(\frac {T} {|v w|}\bigg)^{ 2\gamma}  \bigg\} ,
	\end{equation} 
for any integer $\gamma \geqslant 0$.
\end{lem}

It is then clear that $f (w; v)$ is of rapid decay in the first variable $w$ by Lemma \ref{lem: bounds for f}. 
Next we consider the Fourier transform: \begin{align*}
	\widehat{f} (u ; v) = \viint f (w; v) e [-  u w] \nd w. 
\end{align*}

\begin{lem}\label{lem: bounds Fourier}
 We have the formula
 \begin{align}\label{10eq: f, Fourier}
 	\widehat{f} (u; v)  = - \frac 1 {\pi^2} \frac {\partial^2} {  \partial u  \partial \widebar{u} } \viint \eta (|z|) k (\pi u z) e  [-   { v z}    ] \frac {  \nd z} {|z|^2}, 
 \end{align}
\begin{align}
	k (z) = \pi T^2 \exp \left( -  T^2 |z|^2  \right), 
\end{align}
and consequently  the uniform bounds
\begin{align}\label{10eq: Fourier, 1}
	 	\widehat{f} (u/\pi ; v) \Lt     \frac {T^2} {|u|^2},   
\end{align} 
\begin{align}\label{10eq: Fourier, 1.1} 
	\widehat{f} (u/\pi ; v) \Lt     \bigg( 1 + \log \bigg( 1+   \frac 2 {T|u| } \bigg)  \bigg)  \frac {T^2} {|v|^2},  
\end{align} 
and for $  |u| >  1/T$, 
\begin{align}\label{10eq: Fourier, 2}
	\widehat{f} (u/\pi ; v) \Lt  T^4 \,  \exp \left( -  T^2 |u/ 2|^2  \right) .  
\end{align}  
\end{lem}

\begin{remark}
	Note that our bound {\rm\eqref{10eq: Fourier, 2}} is stronger than Young's {\rm(8.3)}  in \cite{Young-GL(3)-Special-Points} thanks to the fact that the Gaussian is self-dual under the Fourier transform. Consequently, our later analysis in {\rm\S \ref{sec: dual sum}} will be much simpler.  
\end{remark}



The next simple lemma on Fourier integrals will be useful. 


\begin{lem}\label{lem: Fourier}
Let $A_{\rho}$ denote $\BC \smallsetminus D_{\rho}$.	 Suppose that $ f \in C^{\infty} (\BC)$  is supported on  $A_{\rho}$ and satisfies
	 \begin{align*}
  \frac {\partial^{\valpha+\beta} f (z) } {\partial z^{\valpha} \partial \widebar{z}^{\beta} } 	 
	 \Lt_{  \valpha,   \beta }   \frac 1 { |z|^{\theta + \valpha + \beta}}  
	 \end{align*}
 for some  $\theta > 1$.  Then 
 \begin{align*}
 	\viint_{A_{\rho} }  f (z) e [- v z] \nd z = \frac 1 {(  \pi i |v|)^{2\gamma} } \viint_{A_{\rho} } \Delta^{\gamma} f (z) e [- v z] \nd z, \qquad \Delta = \partial^2 / \partial z  \partial \widebar{z}, 
 \end{align*}
for any integer $\gamma \geqslant 0$. 
\end{lem}

\begin{proof}[Proof of Lemma \ref{lem: bounds for f}]
	For brevity, let $\partial^{\gamma} $ denote any partial derivative of degree $\gamma$ (namely, in the form $\partial^{\valpha+\beta} / \partial z^{\valpha} \partial \widebar{z}^{\beta} $, $\gamma = \valpha + \beta$).   
	
	Recall from  \eqref{9eq: f (w;v)} that  
	\begin{align*}
		f (w; v ) =  {|w |^2}   \viint_{A_{1/2}}  \eta (|z|) h   (   w/z    ) e  [-   { v z}    ]   \frac {\nd z} {|z|^4}  . 
	\end{align*}
As $ h (z) = \exp  ( - |z/T|^2  )$ (see \eqref{5eq: h(w)}), it follows by induction that
$\partial^{\beta} (h (w/z)) $ is a linear combination of all the terms occurring in the expansions of
\begin{align*}
  \bigg(  \frac 1 {T^2}  \bigg( \frac {w} {z} + \frac {\widebar{w}}  {\widebar{z}}\bigg)^2    \bigg)^{\valpha}  \lp\frac 1 {z} + \frac 1 {\widebar{z}}  \rp^{\beta}  h \lp \frac {w} {z} \rp, \qquad \valpha = 1, 2, ..., \beta,  
\end{align*}
and hence
\begin{align*}
	\partial^{\beta} (h (w/z)) \Lt_{\beta} \bigg( 1 +   \frac {|  w/z  |}  {T}   \bigg)^{2\beta} \frac {h(w/z)} {|z|^{\beta}} .
\end{align*}
Consequently,  
\begin{align*}
\Delta^{\gamma } \bigg(\frac {\eta (|z|) h (w/z)} {|z|^4}\bigg) \Lt_{\gamma} \bigg(1 +   \frac {|  w/z  |}  {T} \bigg)^{4\gamma} \frac {h (w/z)} {|z|^{ 2\gamma+4}}  
\end{align*}
for any $z \in A_{1/2}$. By Lemma \ref{lem: Fourier} and the change of variable $ z \ra  w/ T z $, we have
\begin{align*}
	f (w; v)   \Lt \frac {T^{2\gamma + 2}} {|v w|^{2\gamma } } \viint_{\overbar{D}_{  2|w| / T  }}  (1 +     {|  z  |}    )^{4\gamma}  |z|^{2\gamma} \exp   ( - |z|^2   ) \nd z , 
\end{align*}
and hence the bound in \eqref{10eq: bound for f, 1} by trivial estimation (in the polar coordinates).  
\end{proof}

\begin{proof}[Proof of Lemma \ref{lem: bounds Fourier}] 
	By definition, we have 
	\begin{align*}
		 \widehat{f} (u; v) & =       \viint |w|^2 e  [  -  u w] \bigg(\viint  \eta (|z|) h   (   w/z    ) e  [-   { v z}   ]   \frac {\nd z } {|z|^4}\bigg) \nd w \\
		 & = - \frac 1 {\pi^2} \frac {\partial^2} {  \partial u  \partial \widebar{u} }  \viint  \viint  \eta (|z|) h   (   w/z    ) e  [-   { v z}     -  u w]   \frac {\nd z \nd w} {|z|^4}. 
	\end{align*}
Then we obtain \eqref{10eq: f, Fourier} by reversing the order of integrations and a direct evaluation of the $w$-integral. However, the quadruple integral  is {\it not}  (but on the edge) absolutely integrable. This issue may be easily addressed by replacing $|z|^{4}$ by $|z|^{4+\vepsilon}$ in the denominator and letting $\vepsilon \ra 0$ at the end. Alternatively, similar to \cite{Young-GL(3)-Special-Points}, one may apply Lemma \ref{lem: Fourier} ($\gamma = 1$ is enough) to produce extra $1/z$ or $1/ \widebar{z}$ and then apply Lemma \ref{lem: Fourier} in reverse  at the end.  

The bounds \eqref{10eq: Fourier, 1},   \eqref{10eq: Fourier, 1.1}, and \eqref{10eq: Fourier, 2} follow from \eqref{10eq: f, Fourier} by trivial estimation (in the polar coordinates), in the second case with an application of Lemma \ref{lem: Fourier}. For example, note that trivially
\begin{align*}
 \frac {\partial^2 k (  u z) } {  \partial u  \partial \widebar{u} }   \Lt  T^4   (1 + T |uz|   )^2 \cdot |z|^2  k (  u z) . 
\end{align*}  
For  \eqref{10eq: Fourier, 1} and \eqref{10eq: Fourier, 1.1}, except for a term containing $1/|z|^2$ in the second case (this yields the logarithmic term),  we may as well extend the integration onto the whole complex plane.  
\end{proof}

\section{Proof of Propositions \ref{prop: Z} and \ref{prop: S}}
\label{sec: proof of props}

\subsection{Treating the Zero Frequency} \label{sec: zero frequency}
Let us first prove Proposition \ref{prop: Z}. Set $ \psi (x) = \eta (1/x) -1 $. Note that $\psi' (x)$ is supported on $[1,2]$. 
We have
\begin{align*}
f (w; 0) = |w|^2	\viint  \eta (|z|) h   (   w/z    )   \frac {\nd z} {|z|^4}   =   \viint   h   (   z   )     {\nd z}   + |w|^2 \viint  \psi (|z |) h   (  w z   )     {\nd z}   .
\end{align*}
Recall from \eqref{5eq: h(w)} that $ h (z) = \exp  ( - |z/T|^2  )$, so the first integral is $\pi T^2$ and hence  contributes 
\begin{align*}
	\pi^3 T^2  \sum_{|c| \leqslant X}  \frac 1 {|c |^4}	\bigg( \sum_{n}    {a}_{(n)}   S(n, 0; c)\bigg)^2 = \pi^3   \varSigma (\boldsymbol{a}) + O \bigg(\frac {T^2 N^{2+\vepsilon} } {X^2} \|\boldsymbol{a}_{N} \|_2^2  \bigg), 
\end{align*}
where  the tail of the sum has been estimated trivially by  \eqref{5eq: Ramanujan}.  
In the polar coordinates, the second integral may be reformulated as
\begin{align*}
 2\pi |w|^2  	\int_1^{2}  \psi  (     x    ) x \exp \left(- |wx/T|^2 \right) \nd x  =  \pi T^2    \int_1^{2} \psi'  (  x  )   \exp \left(- |wx/T|^2 \right) \nd x   , 
\end{align*}
 by partial integration.
Thus one is reduced to estimating the sum 
\begin{align*}
T^2 \sum_{|c| \leqslant X}  \frac 1 {|c |^4}	\mathop{\sum\sum}_{m, n}   {a}_{(m)}  {a}_{(n)} S (m, 0; c) S(n, 0; c) \exp \left( - \Big|\pi \frac {m-n} {c}\Big|^2 \frac {x^2} {T^2} \right) , 
\end{align*}
for any $ 1 \leqslant x \leqslant 2 $. Then separate the variables $m$ and $n$ by 
\begin{align*}
	\exp \left(-   |w x / T|^2 \right) = \frac {1}  {\pi}   \viint \exp  \left(  2 i \Re ( wxz /T)  -   |  z |^2\right) \nd z , 
\end{align*}
so that the sum above is turned into
\begin{align*}
	\frac {T^2} {\pi} \viint \exp (-   |  z |^2) \sum_{|c| \leqslant X}  \frac 1 {|c |^4} \bigg| \sum_{n} a_{(n)} S(n, 0; c) e \Big[ \frac {  n x z  } {c T} \Big] \bigg|^2 \nd z. 
\end{align*}
Next open the Ramanujan sum and  apply Cauchy--Schwarz to bound this by
\begin{align*}
	\frac {T^2} {\pi} \viint \exp (-   |  z |^2) \sum_{|c| \leqslant X}  \frac 1 {|c |^2} \sumx_{   \valpha  (\mathrm{mod} \, c) }   \bigg| \sum_{n} a_{(n)} e \Big[ \frac {\valpha n} {c} \Big] e \Big[ \frac {  n x z  } {c T} \Big] \bigg|^2 \nd z, 
\end{align*}
 so that   \eqref{11eq: hybrid LS, 2} in Corollary \ref{cor: hybrid LS} immediately yields the bound  $O  \big(  T^2 (X^2  + T^2)  N^{\vepsilon}  \|\boldsymbol{a}_{N} \|_2^2 \big)  = O  \big(  T^4 N^{\vepsilon} \|\boldsymbol{a}_{N} \|_2^2 \big)$.  Thus the proof of Proposition   \ref{prop: Z} is completed.
 
  \subsection{Treating the Dual Sum} \label{sec: dual sum}
  Now we consider the dual sum $S (\boldsymbol{a}; X) $ as given in \eqref{9eq: S (a;X)} and prove Proposition \ref{prop: S}.  
  
First of all, the bound \eqref{10eq: bound for f, 1} for $f (w; v)$ in Lemma \ref{lem: bounds for f} suggests that we may assume in practice that $ |q| \leqslant |c|   N^{\vepsilon}$ as in our case $v = q/c$.  By Fourier inversion, 
\begin{align*}
\pi^2 	f \Big(   \pi \frac {m-n} { c}; \frac {q} {c} \Big) = \viint \widehat{f} \Big(\frac u \pi; \frac {q} {c} \Big) e \Big[ \frac {m-n} {c} u  \Big]  \nd u, 
\end{align*} 
and the bound \eqref{10eq: Fourier, 2} for $\widehat{f} (u/\pi; v)$ in Lemma \ref{lem: bounds Fourier} suggests that we may restrict the integration on the disc $\overbar{D}_{  N^{\vepsilon}/T}$.  
It follows that up to an error $O \big( \|\boldsymbol{a}_{N} \|_2^2 \big)$, say, the dual sum $S (\boldsymbol{a}; X)$ is equal to
\begin{equation*}
\begin{split}
    \sum_{|c| \leqslant X} \frac 1 {|c|^4}	\sum_{0 < |q| \leqslant |c|   N^{\vepsilon}     }    \viint_{\overbar{D}_{  N^{\vepsilon}/T}}   \widehat{f} \Big(\frac u \pi; \frac {q} {c} \Big)  \bigg|  \sum_{ n}    {a}_{(n)}  S (n, q; c) e \Big[ \frac {n} {c} u  \Big] \bigg|^2 \nd u  . 
\end{split}
\end{equation*}
 Now \eqref{8eq: S(a;X)=S0(a;X)} and \eqref{8eq: S0(a;X)} are a direct consequence of  the bound \eqref{10eq: Fourier, 1.1}   for $\widehat{f} (u/\pi ; v)$ in Lemma \ref{lem: bounds Fourier}; one needs to bound trivially the logarithmic term on a very small disc $\overbar{D}_{1/N^{4+\vepsilon}}$, say, so that the error is still $O \big( \|\boldsymbol{a}_{N} \|_2^2 \big)$. 

\delete{Next we wish to shrink the integral domain $\overbar{D}_{\rho N^{\vepsilon}}$   into $\overbar{D}_{\rho}$.  To this end, we apply the uniform bound $ \widehat{f}  (  u /\pi;   {q}/ {c}  ) = O \big(I  (K+P)^4 \big)$ on $\overbar{D}_{\rho N^{\vepsilon}} \smallsetminus \overbar{D}_{\rho}$, extend the sum over $q$ to $O(I^2 N^{\vepsilon})$ many complete sums modulo $c$,  open the square and the Kloosterman sums, and execute the summation over $q$. It follows that the error is bounded by 
\begin{align*}
	I^3   (K+P)^4 N^{\vepsilon}  \sum_{|c| \leqslant X} \frac 1 {|c|^2}	     \viint_{\overbar{D}_{\rho N^{\vepsilon}}   } \, \sumx_{\valpha (\mod c)}   \bigg|  \sum_{ n}    {a}_{(n)}  e \Big[  \frac {\valpha n} {c} \Big]  e \Big[ \frac {n} {c} u  \Big] \bigg|^2 \nd u  ,
\end{align*} 
and an application of  \eqref{11eq: hybrid LS, 2}  yields the error  bound $ O \big(I^3   (K+P)^4 N^{\vepsilon} \|\boldsymbol{a}_{N} \|_2^2 \big) $. Note that $ X \rho \leqslant 1$ by our assumption \eqref{8eq: X}.}


It is left to prove that $ S_0   (\boldsymbol{a}; X)$ has the bound in  \eqref{8eq: S(a;X)}. To this end,  in \eqref{8eq: S0(a;X)} we remove the $1/|q|^2$ (use trivially $1/|q|^2 \leqslant 1$), extend the sum over $q$ to $O( N^{\vepsilon})$ many complete sums modulo $c$,  open the square and the Kloosterman sums, and execute the summation over $q$.  It follows that
\begin{align*}
	S_0 (\boldsymbol{a}; X) & \Lt T^4 N^{\vepsilon} \sum_{|c| \leqslant X}    \viint_{\overbar{D}_{  N^{\vepsilon} /T }   } \, \sumx_{\valpha (\mod c)}   \bigg|  \sum_{ n}    {a}_{(n)}  e \Big[  \frac {\valpha n} {c} \Big]  e \Big[ \frac {n} {c} u  \Big] \bigg|^2 \nd u , 
\end{align*}
and an application of  \eqref{11eq: hybrid LS, 2}  in Corollary \ref{cor: hybrid LS} yields   $ O \big(  T^4 X^2   N^{\vepsilon} \|\boldsymbol{a}_{N} \|_2^2 \big)   $.


\delete{Next, we partition the  integral region into  $D_{1/(K+P)}$ and  $O (\log N)$ many dyadic annuli $A_{[\rho, 2 \rho]} $, and accordingly write (split the right-hand side of \eqref{10eq: S total}) 
\begin{align*}
S (\boldsymbol{a}; X) =   S_{0} (\boldsymbol{a}; X) + \sum S_{\rho}  (\boldsymbol{a}; X) +   O \big( \|\boldsymbol{a}_{N} \|_2^2 \big). 
\end{align*}
}

\delete{
For $ S_{\infty} (\boldsymbol{a}; X) $ the integration on $A_{[X^{1+\vepsilon}, \infty)}$ discards all the off-diagonals (to see this, insert an auxiliary cut-off function $\eta (u/X^{1+\vepsilon})$, open the square,  and apply Lemma \ref{lem: Fourier}), and hence 
\begin{align*}
	S_{\infty} (\boldsymbol{a}; X) \Lt 
\end{align*}
by the bound \red{????????} for $ \widehat{f}  (u /\pi;   {q} / {c}  )$ and the Weil bound for $S (n,q;c)$. }

\section{Proof of Proposition \ref{prop: C(a)}} 

On the  identities or asymptotics in \eqref{6eq: after Kuz}, \eqref{7eq: P=Q}, \eqref{7eq: Q(a)=Q(a;X)}, \eqref{8eq: Q(a)=Z(a)+S(a)}, \eqref{9eq: Z, asymp}, and \eqref{8eq: S(a;X)}, we have established 
\begin{equation*}
\begin{split}
	\SC  (\boldsymbol{a}) + \SE    (\boldsymbol{a}) =  \frac {1} {32 } \varSigma (\boldsymbol{a}) + O \bigg(   \bigg( T N^2 + \frac { N^3} {T^2}  + \frac {T^2 N^2} {X^2} + T^4 X^2  \bigg) N^{\vepsilon} \|\boldsymbol{a}_{N} \|_2^2 \bigg),
\end{split}
\end{equation*}
for any $ 1\leqslant X \leqslant T$. Then \eqref{1eq: C(a)+E(a)}  follows if we choose  $	X = \min \big\{    {\sqrt{N/ T}}  , \allowbreak T \big\}$. 

\section{Proof of Proposition \ref{prop: E(a)}} 

In this section, we investigate the Eisenstein $\SE (\boldsymbol{a}) $ as defined in \eqref{2eq: sum, Eis} and prove its asymptotic formula in \eqref{2eq: asymptotic E(a)}. 

We start with the Ramanujan identity (\cite[(2.18)]{B-Mo}):
\begin{align*}
\frac {\sigma_{1-s,   p}  (n) } {\zeta (s,  p)}  = 	\sum_{(c) \subset \CaloO } \frac {S (n, 0; c) \vchi_{ 4 p} (c)}  {|c|^{2s}}    , \qquad \text{($\Re (s) > 1$)} ,
\end{align*} 
where 
\begin{align}\label{13eq: sigma}
	 \sigma_{s, p} ( n ) = \sum_{(d) | (n)} \vchi_{ 4p} (d) |d|^{2s} . 
\end{align}
For $\Re (s)=1$, by the same argument of \S 13.3 by Lemma 13.2 in \cite{Titchmarsh-Riemann}, it follows that  
\begin{align}\label{13eq: Ramanujan}
	\frac { \sigma_{1-s,   p}  (n) } {\zeta (s, p)}   = \sum_{(c): \, \log |c| \leqslant Y^{  \vepsilon} } \frac {S (n, 0; c) \vchi_{ 4 p} (c)}  {|c|^{2s}} + O_{\vepsilon} (1/Y),   
\end{align}
as long as $|\Im ( s)| \leqslant Y$ (no constraint on $p$ is needed).

As $h (\vkappa, p) = \exp \left( - |i\vkappa + p|^2/ T^2 \right)$ and $   \vchi_{ i\vkappa,   p} (z) = \exp \left(   i \vkappa \log |z| +   i p \arg (z)  \right) $, we have
\begin{align}\label{13eq: integral}
  \viint_{ \BSA } h (2\vkappa, 4 p)  \vchi_{2i\vkappa, 4 p} (z)     \nd \mu (\vkappa, p ) = k (\log |z|) \, \theta (2 \arg (z) ), 
\end{align}
where, similar to \eqref{4eq: defn of h, theta}, 
\begin{align} \label{13eq: defn of h, theta} 
	k   (r )   =  \frac{\sqrt{\pi} T} {2} \exp \left(     -  (T r / 2)^2   \right), \quad 	 
	\theta  (\omega) =  \frac {\sqrt{\pi} T} 4 \sum  \exp \left(     - (T (\omega + \pi p )/ 4 )^2 \right)    . 
\end{align} 

By \eqref{1eq: tau s} and \eqref{13eq: sigma},  
$$\sigma_{2s, 2 p} (n) = \tau_{s, p} (n) 	\vchi_{2s,  4 p} (n) ,$$  
so    \eqref{2eq: sum, Eis} may be expanded and rewritten  as 
\begin{align}	\label{13eq: sum, Eis}	
	\SE (\boldsymbol{a}) =	 \frac 1 {\pi} \mathop{\sum\sum}_{m, n} a_{(m)} a_{(n)}	\viint_{ \BSA } h (2\vkappa, 4 p) \frac {\sigma_{2i\vkappa, 2p} (m/n) } { | \zeta (1+2 i \vkappa, 2 p)   |^2  }    \nd \mu (\vkappa, p ) . 
\end{align}
 Now let $s = 1 \pm 2 i \vkappa$,  $Y = T^{2 }$ in  \eqref{13eq: Ramanujan} and $z = b/c$ in \eqref{13eq: integral}, it follows     that
\begin{align*}
\SE (\boldsymbol{a}) =	\frac 1 {\pi} \mathop{\sum\sum}_{m, n} a_{(m)} a_{(n)}  X (m, n) + O \big(N^{2+\vepsilon} \|\boldsymbol{a}_{N} \|_2^2 \big),
\end{align*}
where 
\begin{align*}
	X (m, n)  = \mathop{\mathop{\sum\sum}_{(b),  (c)}}_{\, \log  |b|,  \log |c| \leqslant T^{\vepsilon} }    \frac {S(m, 0; b) S(n, 0; c)} {|bc|^2 } \cdot   {k  (\log |b/c|  ) \, \theta  (2 \arg (b/c)  )}   . 
\end{align*} 
Next, according to $(b) = (c)$ or  $(b) \neq (c)$ in the sum $X (m, n)$,  we split 
\begin{align*}
	\SE (\boldsymbol{a} ) = \SE_0 (\boldsymbol{a} ) + \SE_{\flat} (\boldsymbol{a} ). 
\end{align*}
Clearly, 
\begin{align*}
	\SE_0 (\boldsymbol{a} ) = \frac {k (0) \theta (0)} {4 \pi}   \mathop{\sum\sum}_{m, n} a_{(m)} a_{(n)} \sum_{  \log |c| \leqslant T^{  \vepsilon}}    \frac {S(m, 0; c) S(n, 0; c)} {|c|^4 }, 
\end{align*}
and if we approximate  $k (0) \theta (0) $ by $ \pi T^2 / 8 $ and estimate the tail of the $c $-sum by \eqref{5eq: Ramanujan}, then it follows that
\begin{align}\label{13eq: E0}
	\SE_0 (\boldsymbol{a} ) = \frac 1 {32} \varSigma (\boldsymbol{a}) + O  \big(N^{2+\vepsilon} \|\boldsymbol{a}_{N} \|_2^2 \big). 
\end{align}
Since, in view of \eqref{13eq: defn of h, theta},  $k (r) $ and $\theta (\omega)$ are of exponential decay, we have
\begin{align*}
\SE_{\flat} (\boldsymbol{a}) \Lt T^2	{\mathop{\mathop{\mathop{\sum\sum}_{(b)\neq (c)}}_{\, \log  |b|,  \log |c| \leqslant T^{\vepsilon} } }_{ |\log |b/c|| , |\arg (b/c)|  \leqslant T^{\vepsilon} / T }}  \mathop{\sum\sum}_{m, n}  | a_{(m)} a_{(n)}  | \frac{|S(m,0; b) S(n,0; c)|} {|bc|^2} +  N^{2+\vepsilon}  \|\boldsymbol{a}_{N} \|_2^2  ,
\end{align*}
then,  by the AM--GM inequality and by symmetry,   the sum is bounded by
\begin{align*}
T^2 \sum_{(c)} \frac 1 {|c|^4}  \sum_{(b) \neq (c)}    \bigg(  \sum_{ n }   \big| a_{(n)} S (n, 0; c) \big|   \bigg)^2 , 
\end{align*} 
where the summation over $(b)$ and $(c)$ of course is subject to the conditions inherited from the above. 
Note that the number of ideals  $(b) $ is bounded  by $O \big( | c|^2   T^{\vepsilon} / T^2 \big)$, since one may as well extend  the range of $b$ onto the punctured  disc $\overbar{D}_{  |c|     T^{\vepsilon} / T  }  (c) \smallsetminus \{c\} $ (in the first quadrant). Thus an estimation by \eqref{5eq: Ramanujan} yields
 \begin{align}\label{13eq: Ef}
\SE_{\flat} (\boldsymbol{a}) 	\Lt   N^{2+\vepsilon} \|\boldsymbol{a}_{N} \|_2^2 . 
 \end{align}
By \eqref{13eq: E0} and \eqref{13eq: Ef} we complete the proof of Proposition \ref{prop: E(a)}.

\section{Proof of Theorem \ref{prop: C(a), 2}}

In view of \eqref{8eq: S(a;X)=S0(a;X)} and \eqref{8eq: S0(a;X)} in Proposition \ref{prop: S}, we may refine \eqref{1eq: C(a)+E(a)} into: 
\begin{equation*}
	\begin{split}
		\SC  (\boldsymbol{a}) + \SE    (\boldsymbol{a}) =  \frac {1} {32 } \varSigma (\boldsymbol{a}) + S_0 (\boldsymbol{a}; X) + O \bigg(   \bigg( T N^2 + \frac { N^3} {T^2}  + \frac {T^2 N^2} {X^2}    \bigg) N^{\vepsilon} \|\boldsymbol{a}_{N} \|_2^2 \bigg),
	\end{split}
\end{equation*} 
while we have just proven \eqref{2eq: asymptotic E(a)}: 
\begin{align*} 
	\SE  (\boldsymbol{a}) =	 \frac {1} {32 } \varSigma (\boldsymbol{a}) +  O \big(  N^{2+\vepsilon} \|\boldsymbol{a}_{N} \|_2^2 \big) ,
\end{align*}
and hence  Theorem \ref{prop: C(a), 2} is a direct consequence.

\def\cprime{$'$}

\end{document}